\documentclass{amsart}

\usepackage{amssymb,amsmath,color,ytableau,youngtab,graphicx,float}
\usepackage{amsthm}
\usepackage{url}
\usepackage{tikz}

\definecolor{red}{rgb}{1,0,0}
\definecolor{blue}{rgb}{.2,.2,.8}

\def\PED{\mathcal{PED}}
\def\POD{\mathcal{POD}}

\newtheorem{theorem}{Theorem}[section]
\newtheorem{corollary}[theorem]{Corollary}
\newtheorem{proposition}{Proposition}

\theoremstyle{definition}

\newtheorem{example}{Example}
\newtheorem{remark}{Remark}

\newcommand{\ds}{\displaystyle}
\allowdisplaybreaks

\begin{document}

\title[]{PED and  POD partitions: combinatorial proofs of recurrence relations}
\author{Cristina Ballantine}\address{Department of Mathematics and Computer Science\\ College of the Holy Cross \\ Worcester, MA 01610, USA \\} 
\email{cballant@holycross.edu} 
\author{Amanda Welch} \address{Department of Mathematics and Computer Science\\ Eastern Illinois University \\ Charleston, IL 61920, USA \\} \email{arwelch@eiu.edu}

\maketitle

\begin{abstract}PED partitions are partitions with even parts distinct while odd parts are unrestricted. Similarly, POD partitions have distinct odd parts while even parts are unrestricted. In \cite{M17} several recurrence relations for the number of PED partitions of $n$ are proved analytically. They are similar to the recurrence relation for the number of partitions of $n$ given by Euler's pentagonal number theorem. We provide combinatorial proofs for all of these theorems and also for the pentagonal number theorem for PED partitions proved analytically in \cite{FGK}, which motivated the theorems in \cite{M17}. Moreover, we prove combinatorially a recurrence for POD partitions given in \cite{BM21}, Beck-type identities involving PED and POD partitions, and several other results about PED and POD partitions. 
\end{abstract}



\section{Introduction}

A partition $\lambda$ of $n$ is a non-increasing sequence $\lambda= (\lambda_1, \lambda_2, \ldots, \lambda_\ell)$ of positive integers that add up to $n$. We refer to the integers $\lambda_i$ as the parts of $\lambda$.  As usual, we denote by $p(n)$  the number of  partitions of $n$. Note that $p(x)=0$ if $x$ is not a non-negative integer, and since the empty partition $\emptyset$ is the only partition of $0$, we have that $p(0)=1$. 

In this article we consider partitions in which parts of fixed parity are distinct: PED partitions have no repeated even parts and  POD partitions have no repeated odd parts.  We denote by $ped(n)$, respectively $pod(n)$,  the number of PED, respectively POD, partitions of $n$.  

One of the most celebrated theorems in the theory of partitions is Euler's pentagonal number theorem, which states that, for $n\geq 1$, $$\sum_{j=-\infty}^\infty (-1)^kp(n-k(3k+1)/2)=0.$$

Fink, Guy and Krusemeyer \cite{FGK} proved an analogous result for PED partitions.

\begin{theorem}\label{FGK} For $n\geq 0$, $$\sum_{j=-\infty}^\infty (-1)^{j} ped(n-j(3j+1)/2)=\begin{cases} (-1)^k & \text{ if } n=2k(3k+1) \text{ for some } k \in \mathbb Z,\\ 0&  \text{ otherwise. }  \end{cases}$$
\end{theorem}
Inspired by this theorem, Merca \cite{M17} proved analogous recurrences for $ped(n)$  involving triangular and square numbers. 
Let $T_k=k(k+1)/2$ denote the $k^{th}$ triangular number. 
\begin{theorem} \label{Merca 1.1} For $n\geq 0$, $$\sum_{j\geq 0} (-1)^{\lceil j/2\rceil} ped(n-T_j)=\begin{cases} 1 &  \text{ if } n=2T_k \text{ for some } k \geq 0,\\ 0 &  \text{ otherwise. }  \end{cases}$$
\end{theorem}

\begin{theorem} \label{Merca 1.2} For $n\geq 0$, $$\sum_{j=-\infty}^\infty (-1)^{j} ped(n-2j^2)=\begin{cases} 1&  \text{ if } n=T_k \text{ for some } k \geq 0,\\ 0&  \text{ otherwise. }  \end{cases}$$
\end{theorem}

In the  articles cited, Theorems \ref{FGK}, \ref{Merca 1.1}, and \ref{Merca 1.2} are proved using generating functions. We give combinatorial proofs of these theorems. 

In \cite{BM21}, analogous results were derived analytically for $pod(n)$. We denote by $Q_k(n)$ the number of distinct partitions of $n$ into parts $\not \equiv k \bmod 4$. 

\begin{theorem}\label{pod-rec}
	For $n\geq 0$ the following hold.
	\begin{enumerate}
		\item[(i)] $\displaystyle{Q_0(n) = pod(n)+2 \sum_{k=1}^{\infty} (-1)^{k}\, pod\big(n-4k^2\big)}$
		\item[(ii)] $\displaystyle{Q_2(n) = \sum_{k=0}^{\infty} (-1)^{k(k+1)/2}\, pod\big(n-k(k+1)\big)}$
	\end{enumerate}
\end{theorem}
In \cite{BM21}, the authors also proved Theorem \ref{pod-rec}\,(ii) combinatorially. Here we give a combinatorial proof of Theorem \ref{pod-rec}\,(i).

We also give combinatorial proofs of the remaining theorems in \cite{M17} listed below. We prove the following theorem relating $ped(n)$ and $pod(n)$. 

\begin{theorem}\label{Merca 4.1} For any non-negative integer $n$, we have $$ped(n)=\sum_{k=0}^\infty pod(n-2T_k).$$
 \end{theorem} 
We then prove another recurrence for $pod(n)$ involving triangular numbers. 

\begin{theorem}\label{Merca 4.3}  For $n\geq 0$, $$\sum_{j=0}^\infty (-1)^{T_j}pod(n-T_j)=\begin{cases}1 & \text{ if } n=0,\\ 0 & \text{ if } n>0. \end{cases}$$ 
\end{theorem}

We give a combinatorial proof of Merca's theorem relating $ped(n)$ and $\overline p(n)$, the number of overpartitions of $n$. Overpartitions were introduced in \cite{CL}. They are partitions in which the first occurrence of a part may be overlined. 

 \begin{theorem} \label{Merca 5.1} For $n\geq 0$,  $$ped(n)=\sum_{k\geq 0}\overline p\left(\frac{n}{2}-\frac{T_k}{2}\right),$$ where $\overline p(x)=0$ if $x$ is not an integer. 
\end{theorem}

Let $o(n)$ denote the number of partitions of $n$ into odd parts and denote by $o_{e-o}(n)$  the number of partitions of $n$ with an even number of odd parts  minus  the number of partitions of $n$ with an odd number of odd parts (in each case there are no even parts). We give a combinatorial theorem of the following result stated as  Corollary 5.4 in \cite{M17}. 

\begin{theorem}\label{C 5.4} For $n\geq 0$, $$\sum_{j=0}^\infty o_{e-o}(n-T_j)=\begin{cases} (-1)^k  & \text{ if } n=2k(3k+1) \text{ for some }  k \in \mathbb Z\\ 0& \text{ otherwise. }\end{cases}$$
\end{theorem}

Combining Theorem \ref{C 5.4} with Theorem \ref{FGK} we obtain a combinatorial proof of  Merca's Theorem 5.3. 

\begin{corollary} \label{Merca 5.3} For $n\geq 0$, $$\sum_{j=-\infty}^\infty (-1)^j ped(n-j(3j+1)/2)=\sum_{j=0}^\infty o_{e-o}(n-T_j).$$
\end{corollary}

In the last section of the article we prove Beck-type theorems for $ped(n)$ and $pod(n)$. As we will explain in the next section, $ped(n)$ is equal to the number of partitions of $n$ into parts $\not \equiv 0 \bmod 4$, and $pod(n)$ is equal to the number of partitions of $n$ into parts $\not \equiv 2 \bmod 4$. The Beck-type identities give combinatorial interpretations for the excess in the number of parts in all partitions of $n$ into parts $\not \equiv 0 \bmod 4$ over the number of parts in all PED partitions of $n$; and also for the excess in the number of parts in all partitions of $n$ into parts $\not \equiv 2 \bmod 4$ over the number of parts in all POD partitions of $n$. These are analogous to the interpretations given in \cite{A-Beck} for the excess in the number of parts in all partitions of $n$ into odd parts over the number of parts in all distinct partitions of $n$.

\section{Preliminaries and Notation } \label{prelim}

If $\lambda$ is a partition of $n$, we write $\lambda\vdash n$. We refer to $n$ as the size of $\lambda$ and also write $|\lambda|=n$. When we work with vectors of partitions, we write $(\lambda, \mu, \cdots)\vdash n$ to mean $|\lambda|+|\mu|+\cdots =n$.
 The length of $\lambda$ is the number of parts of $\lambda$, denoted by $\ell(\lambda)$.   For convenience, we abuse notation and use $\lambda$ to denote either the multiset of its parts or the non-increasing sequence of parts.  We write $a\in \lambda$ to mean the positive integer $a$ is a part of $\lambda$. The number of times $a>0$ appears in $\lambda$ is denoted by $m(a)$ and is called the multiplicity of $a$. 
We use the convention that $\lambda_k=0$ for all $k>\ell(\lambda)$. 

The {Ferrers diagram} of a partition $\lambda=(\lambda_1, \lambda_2, \ldots, \lambda_\ell)$ is an array of left justified boxes such that the $i$th row from the top contains $\lambda_i$ boxes. We abuse notation and use $\lambda$ to mean a partition or its Ferrers diagram.
\begin{example} The Ferrers diagram of $\lambda=(4, 3, 3, 2)$  is shown below.  \medskip

 \begin{center}\small{\ydiagram[*(white)]
{4,3,3,2}}\end{center}
 \end{example}
 
 We define the following operations on partitions. If   $\lambda$ and $\mu$ are partitions, by $\lambda\cup \mu$ and $\lambda\setminus \mu$ we mean the obvious operations on the multisets of parts of $\lambda$ and $\mu$. Moreover $\lambda\setminus \mu$ is only defined if $\mu \subseteq \lambda$ as multisets. For a positive integer $k$, we denote by $k\lambda$ the partition whose parts are the parts of $\lambda$ multiplied by $k$. If all parts of $\lambda$ are divisible by $k$, we denote by $\lambda_{/k}$ the partition whose parts are the parts of $\lambda$ divided by $k$.
 
 We often write a partition $\lambda$ as $(\lambda^e, \lambda^o)$, where $\lambda^e$, respectively $\lambda^o$, consists of the even, respectively odd, parts of $\lambda$. 

We denote by calligraphy style capital letters the set of partitions enumerated by the function denoted by the same letters. For example, $\mathcal P(n)$ is the set of unrestricted partitions of $n$, which are enumerated by $p(n)$, and $\PED (n)$ is the set of partitions of $n$ in which even parts are distinct and thus $|\PED(n)|=ped(n)$. If the variable $n$ is omitted, we mean the set of all partitions with the obvious restrictions. For example, $\PED=\ds \cup_{n\geq 0}\PED(n)$.   When necessary, we will clarify the notation. 

For a  thorough and detailed introduction to the theory of integer partitions, we refer the reader to \cite{A98}.

We first derive a few basic facts about PED and POD partitions. Most are well-known and can be found for example in \cite{A09, M17, BM21} and the references therein. 

We use the Pochhammer symbol notation
\begin{align*}
	& (a;q)_n = \begin{cases}
	1 & \text{for $n=0$,}\\
	(1-a)(1-aq)\cdots(1-aq^{n-1}) &\text{for $n>0$;}
	\end{cases}\\
	& (a;q)_\infty = \lim_{n\to\infty} (a;q)_n.
	\end{align*}
	Throughout, we assume $|q|<1$ so that all series converge absolutely. 
	
	The generating function for $ped(n)$ is given by $$\sum_{n\geq 0} ped(n)q^n=\frac{(-q^2;q^2)_\infty}{(q;q^2)_\infty}.$$
This can be rewritten as  \begin{equation}\label{4-reg}\sum_{n\geq 0} ped(n)q^n=\frac{(-q^2;q^2)_\infty}{(q;q^2)_\infty}\cdot \frac{(q^2;q^2)_\infty}{(q^2;q^2)_\infty}=\frac{(q^4;q^4)_\infty}{(q;q)_\infty}=\sum_{n\geq 0} b_4(n)q^n,\end{equation} 
where $b_4(n)$ is the number of $4$-regular partitions of $n$, i.e., partitions with no parts congruent to $0\bmod 4$.

One can also see  combinatorially that $ped(n)=b_4(n)$.  Start with  $\lambda\in \PED(n)$  and, similar to Glaisher's transformation,  split each  part of the form $2^kc$ with $k\geq 2$ and $c$ odd into $2^{k-1}$ parts equal to $2c$ to obtain    a partition $\mu \in \mathcal B_4(n)$.  Thus, each part of $\lambda$ congruent to $0\bmod 4$ is split into equal parts congruent to $2\bmod 4$. The inverse transformation is defined by starting with $\mu\in \mathcal B_4(n)$ and repeatedly merging equal even parts until all even parts are distinct to obtain  a partition $\lambda\in \PED(n)$.   

Similarly, the generating function for $pod(n)$ is given by$$\sum_{n\geq 0} pod(n)q^n=\frac{(-q;q^2)_\infty}{(q^2;q^2)_\infty},$$ which can be rewritten as  
\begin{equation}\label{pod-mod 4}\sum_{n\geq 0} pod(n)q^n=\frac{(-q;q^2)_\infty}{(q^2;q^2)_\infty}\cdot \frac{(q;q^2)_\infty}{(q;q^2)_\infty}=\frac{(q^2;q^4)_\infty}{(q;q)_\infty}=\sum_{n\geq 0} p_2(n)q^n,\end{equation} where 
 $p_2(n)$ is the number of partitions of $n$ with no parts congruent to $2\bmod 4$. 
 
To show combinatorially that $pod(n)=p_2(n)$, start with $\lambda\in \POD(n)$ and split each part congruent to $2\bmod 4$ into two equal odd parts to obtain a partition in $\mathcal P_2(n)$. For the inverse, start with a partition $\mu \in \mathcal P_2(n)$. For each repeated odd part $c$ of $\mu$ with multiplicity $m(c)\geq 2$, merge $\lfloor m(c)/2\rfloor$ pairs of parts equal to $c$ to obtain $\lfloor m(c)/2\rfloor$ parts equal to $2c$. The obtained partition is in $\POD(n)$.

Next, we observe that $ped(n)- pod(n)\geq 0$ for all $n\geq 0$ and we give a combinatorial interpretation of the difference.  
\begin{proposition}  Let $n\geq 0$. Then $b_4(n)-p_2(n)$ equals the number of $4$-regular partitions of $n$ such that the number of parts equal to $1$ is less than twice the number of even parts. 
\end{proposition}

\begin{proof} We define a bijection from $\mathcal P_2(n)$ to the set of $4$-regular partitions of $n$ such that  $m(1)$ is at least twice the number of even parts. 

If $\lambda=(\lambda^e, \lambda^o)\in \mathcal P_2(n)$, then all parts of $\lambda^e$ are divisible by $4$. Subtract $2$ from each  part in $\lambda^e$ and  insert $2\ell(\lambda^e)$ parts equal to $1$ to obtain a $4$-regular partition $\mu=(\mu^e, \mu^0)$ with  $m(1)\geq 2\ell(\mu^e)$. 

Conversely, if $\mu=(\mu^e, \mu^o)$ is a $4$-regular partition of $n$ with  $m(1)\geq 2\ell(\mu^e)$,  remove $2\ell(\mu^e)$ parts equal to $1$ and add $2$ to each even part to obtain a partition $\lambda\in \mathcal P_2(n)$.
\end{proof}

 We use the bijection between $\mathcal B_4(n)$ and $\PED(n)$ described above to express $ped(n)-pod(n)$ as the cardinality of a subset of $\PED(n)$. If $\lambda=(\lambda^e, \lambda^o)\in \PED(n)$ is obtained from  $\mu=(\mu^e, \mu^o)\in \mathcal B_4(n)$, then $$\ell(\mu^e)=\sum_{a\in \lambda^e}2^{val_2(a)-1},$$ where $val_2(a)$ is the $2$-adic valuation of $a$, i.e., $val_2(a)=k$ such that $a=2^kc$ with $c$ odd. Moreover, $\lambda^o=\mu^o$. 
Thus, $$ped(n)-pod(n) =\left|\left\{\lambda\in \PED(n) \mid m(1)< \displaystyle \sum_{a\in \lambda^e}2^{val_2(a)}\right\}\right|.$$
For any set of partitions $\mathcal A(n)$, let $$a_{e-o}(n):=|\{\lambda \in \mathcal A(n)\mid \ell(\lambda) \text{ even}\}|-|\{\lambda \in \mathcal A(n)\mid \ell(\lambda) \text{ odd}\}|.$$ 

By a distinct partition we mean a partition in which no part is repeated. By an odd partition we mean a partition in which all parts are odd. 
Recall that, for $k \in \{0,2\}$, we denote by $Q_k(n)$ the number of {distinct} partitions   with no parts congruent to $k\bmod 4$.

 In \cite[Theorem 1.2]{BM21}  it is shown that $$p_{2,e-o}(n)=(-1)^nQ_0(n)$$  and $$b_{4,e-o}(n)=(-1)^nQ_2(n).$$  
 Here, we give a somewhat analogous result. 
 \begin{theorem}\label{ped_e-o} For all $n\geq 0$, $$ped_{e-o}(n)=Q_{2,e-o}(n)$$ and $$pod_{e-o}(n)=Q_{0,e-o}(n).$$
 \end{theorem}
 \begin{proof}

It is fairly straightforward to see that  

 \begin{align*}\sum_{n\geq 0} ped_{e-o}(n)q^n& =\frac{(q^2;q^2)_\infty}{(-q;q^2)_\infty}=\frac{(q;q^2)_\infty(-q;q^2)_\infty(q^4;q^4)_\infty}{(-q;q^2)_\infty}\\ & =(q;q^2)_\infty(q^4;q^4)_\infty= \sum_{n\geq 0}Q_{2,e-o}(n) q^n.\end{align*}

Similarly, 

\begin{align*}\sum_{n\geq 0} pod_{e-o}(n)q^n& =\frac{(q;q^2)_\infty}{(-q^2;q^2)_\infty}= (q;q^2)_\infty(q^2;q^4)_\infty= \sum_{n\geq 0}Q_{0,e-o}(n) q^n.\end{align*} For the middle equality, we used Euler's identity $\ds\frac{1}{(q;q^2)_\infty}=(-q;q)_\infty$ with $q$ replaced by $q^2$. 
 \end{proof} 
 We will  give combinatorial proofs for the identities in Theorem \ref{ped_e-o} at the end of the next section. 
 
\section{Involutions and bijections for partition identities} \label{inv}

To prove combinatorially the theorems stated in  the introduction, we make use of several involutions and bijections from the literature. In this section, we survey these transformations. The descriptions of most of these mappings are fairly involved and we do not reproduce them here. 
To help the reader keep track of the different involutions and bijections,  we adopt the notation $\varphi_*$ for the transformation due to authors with initials $*$. When defining an involution, there will be an exceptional set of partitions on which the involution is not defined. If the domain of the involution is (most of) the set of partitions $\mathcal A(n)$, we denote the exceptional set by $\mathcal E_{\mathcal A}(n)$.  To prove  identities in which a partition $\lambda$  is counted with weight $(-1)^{\ell(\lambda)}$, the involutions described below reverse the parity of the length of partitions. We refer to them as sign reversing involutions. 

\noindent {\bf 1.}  The Kolitsch-Kolitsch transformation $\varphi_{K}$ for the combinatorial proof of \begin{equation}\label{JTP1}(q^4;q^4)_\infty(q^3;q^4)_\infty(q;q^4)_\infty=\sum_{n=0}^\infty(-1)^{T_n} q^{T_n}.
\end{equation} This is a special case of Jacobi's triple product. In the literature, there are many combinatorial proofs of the general form of the Jacobi triple product. We found the proof by Kolitsch-Kolitsch \cite{KK18} to be most intuitive. 

Recall that $\mathcal Q_2(n)$ is the set of distinct partitions of $n$ with parts not congruent to $2 \bmod 4$. Let $\mathcal{E}_{\mathcal Q_2}(n)$ be the subset of $\mathcal Q_2(n)$ defined by 
$$\mathcal{E}_{\mathcal Q_2}(n)=\begin{cases} \{(4i-1, 4(i-1)-1, \ldots, 7,3)\} & \text{ if $n=2i^2+i$ for some   $i>0$},\\  \{(4(i-1)+1,4(i-2)+1 \ldots, 5,1)\}  & \text{ if $n=2i^2-i$ for some   $i>0$}, \\ \{\emptyset\} & \text{ if } n=0, \\ \emptyset & \text{ otherwise}.\end{cases}$$  Note that above $\{\emptyset\}$ means that the given set contains the empty partition. 

 If $n=2i^2\pm i$ for some nonnegative integer $i$, the partition in  $\mathcal E_{\mathcal Q_2}(n)$ has  length $i$.  Moreover,  $2i^2+i=T_{2i}\equiv i \bmod 2$ and $2i^2-i =T_{2i-1}\equiv i \bmod 2$.

Then, if $n>0$, the transformation $\varphi_{K}$ defined by L. W. Kolitsch and S. Kolitsch in \cite{KK18} (with $r=4$ and $s=1$),  is a sign reversing  involution on $\mathcal Q_2(n)\setminus \mathcal{E}_{\mathcal Q_2}(n)$. 
It proves combinatorially that
$${Q}_{2,e-o}(n)=\begin{cases}(-1)^{T_i} & \text{ if $n=T_i$ for some } i\geq 0,\\ 0 & \text{ otherwise. }  \end{cases}$$ 

\medskip

\noindent {\bf 2.} Franklin's transformation $\varphi_F$ for the combinatorial proof of Euler's pentagonal number theorem. 

Let $a(i):=(3i^2+i)/2$, $i\in \mathbb Z$. We denote by  $\mathcal Q(n)$  the set of distinct partitions of $n$ and  let $\mathcal{E}_{\mathcal Q}(n)$ be the subset of $\mathcal Q(n)$ defined by $$\mathcal{E}_{\mathcal Q}(n)=\begin{cases} \{\pi_i^+:=(2i, 2i-1, \ldots, i+1)\} & \text{ if $n=a(i)$ for some   $i>0$},\\  \{\pi_i^-:=(2i-1, 2i-2, \ldots, i)\}  & \text{ if $n=a(-i)$ for some   $i>0$}, \\ \{\pi_0:=\{\emptyset\} \}& \text{ if } n=a(0)=0, \\ \emptyset & \text{ otherwise}.\end{cases}$$  We refer to $\pi_i^+$ and $\pi_i^-$ as pentagonal partitions and notice that each has $i$ parts.  Then, the transformation $\varphi_F$  defined by Franklin (see for example \cite{A98})  is a sign reversing  involution on $\mathcal Q(n)\setminus \mathcal{E}_{\mathcal Q}(n)$. It proves combinatorially that $$Q_{e-o}(n)=\begin{cases}(-1)^{i} & \text{ if $n=a(i)$ for some } i\in \mathbb Z,\\ 0 & \text{ otherwise. }  \end{cases}$$ 

\medskip

\noindent {\bf 3.}  If $n>0$, the Bressoud-Zeilberger transformation $\varphi_{BZ}$ defined in \cite{BZ85} is an involution on $\displaystyle \bigcup_{i=-\infty}^\infty \mathcal P(n-a(i))$ that reverses the parity of $i$. It proves combinatorially that $$\sum_{i=-\infty}^\infty (-1)^i p(n-a(i))=\begin{cases}0 & \text{ if } n>0, \\ 1 & \text{ if } n=0.\end{cases}$$ 
To ease notation, when $\lambda\vdash n -a(i)$, we write $\varphi_{BZ}(\pi, \lambda)$, with $\pi$ the pentagonal partition of size $a(i)$, instead of $\varphi_{BZ}(\lambda)$.

\medskip

\noindent {\bf 4.} The Ballantine-Merca transformation defined in \cite{BM21a}  is a bijection  $$\varphi_{BM}: \bigcup_{k\geq 0} \mathcal P(n-T_k)\to \mathcal{QQ}(n),$$ where $\mathcal{QQ}(n) = \{(\lambda, \mu)\vdash n \mid \lambda, \mu \in \mathcal Q\}$ is the set of distinct partitions in two colors. 

\medskip 

Next we make note of two sign reversing involutions defined by Andrews in \cite{A72} to prove Gauss' theta identities.

\noindent {\bf 5.} The Andrews transformation $\varphi_{A_1}$ for the combinatorial proof  of Gauss' theta identity \begin{equation*}\label{gauss1} \frac{(q^2;q^2)_\infty}{(-q;q^2)_\infty}=  \sum_{n=0}^\infty(-1)^{T_n} q^{T_n}.\end{equation*} The left hand side of the above identity is the generating function for $ped_{e-o}(n)$.
Let $\mathcal{E}_{\PED}(n)$ be the subset of $\PED(n)$ defined by $$\mathcal{E}_{\PED}(n)=\begin{cases} \{((2i+1)^i)\} & \text{ if $n=2i^2+i$ for some   $i>0$},\\  \{((2i-1)^i)\}  & \text{ if $n=2i^2-i$ for some   $i>0$}, \\ \{\emptyset\} & \text{ if } n=0, \\ \emptyset & \text{ otherwise}.\end{cases}$$ Then the transformation $\varphi_{A_1}$ defined by Andrews in \cite{A72} is a sign reversing  involution on $\PED(n)\setminus \mathcal{E}_{\PED}(n)$. It proves combinatorially that $$ped_{e-o}(n)=\begin{cases}(-1)^{T_i} & \text{ if $n=T_i$ for some } i\geq 0,\\ 0 & \text{ otherwise. }  \end{cases}$$

\noindent {\bf 6.} The Andrews transformation $\varphi_{A_2}$ for the combinatorial proof of Gauss' second theta identity $$\frac{(q;q)_\infty}{(-q;q)_\infty}=\sum_{n=-\infty}^\infty (-1)^nq^{n^2}.$$

Let $\overline{\mathcal P}(n)$ be the set of overpartitions of $n$. Then the left hand side of the above identity is the generating function for $\overline p_{e-o}(n)$. Let $\mathcal{E}_{\overline{\mathcal P}}(n)$ be the subset of $\overline{\mathcal P}(n)$ defined by $$\mathcal{E}_{\overline{\mathcal P}}(n)=\begin{cases} \{(\overline m, m, m \ldots, m), (m, m, m \ldots, m)\} & \text{ if $n=m^2$ for some   $m>0$},\\ \{\emptyset\} & \text{ if } n=0, \\ \emptyset & \text{ otherwise}.\end{cases}$$ Then the transformation $\varphi_{A_2}$ defined by Andrews in \cite{A72} is a sign reversing  involution on $\overline{\mathcal P}(n)\setminus \mathcal{E}_{\overline{\mathcal P}}(n)$. It proves combinatorially that $$\overline p_{e-o}(n)=\begin{cases}2(-1)^{m} & \text{ if $n=m^2$ for some } m\geq 0,\\ 1 & \text{ if $n=0$},
\\ 0 & \text{ otherwise. }  \end{cases}$$

\begin{remark} In fact, the transformation $\varphi_{A_2}$ also reverses the parity of the number of overlined parts in overpartitions. 
\end{remark}

\noindent {\bf 7.} Gupta's transformation $\varphi_G$ for the combinatorial proof of $$\frac{1}{(-q^2;q^2)_\infty(q;q^2)_\infty}=(-q;q^2)_\infty.$$

Let $\mathcal Q_{odd}(n)$ be the set of partitions of $n$ with all parts odd and distinct. Gupta \cite{G75}  defined an involution $\varphi_G$ on $\mathcal P(n)\setminus \mathcal Q_{odd}(n)$ that reverses the parity of the number of even parts.  It shows combinatorially   that $$p_e(n,2)-p_o(n,2)=q_{odd}(n),$$ where $p_{e/o}(n,2)$ is the number of partitions of $n$ with an even (respectively odd) number of even parts. 

Gupta's transformation also shows combinatorially that $$p_e(n)-p_o(n)=(-1)^n q_{odd}(n), $$ where $p_{e/o}(n)$ is the number of partitions of $n$ with an even (respectively odd) number of  parts. 

\medskip

\noindent {\bf 8.} Glaisher's bijection (see for example \cite[Section 2.3]{AE}) $\varphi_{Gl}: \mathcal O(n) \to \mathcal Q(n)$ for the combinatorial proof of Euler's identity $$\frac{1}{(q;q^2)_\infty}=(-q;q)_\infty.$$ 

We already alluded to this transformation in the previous section.

\noindent{\bf 9.}  Finally we prove combinatorially that $$(-q;q)_\infty(q;q)_\infty=(q^2;q^2)_\infty.$$ Let $\mathcal Q_{even}(n)$ be the set of partitions of $n$ with all parts even and distinct.  Recall that $\mathcal{QQ}(n)=\{(\lambda, \mu)\vdash n \mid \lambda, \mu \in \mathcal Q\}$. We have the following combinatorial interpretation of the above identity.

\begin{proposition} \label{zeta} For $n\geq 0$ we have $$|\{(\lambda, \mu)\in \mathcal{QQ}(n) \mid \ell(\lambda) \text{ even}\}|-|\{(\lambda, \mu)\in \mathcal{QQ}(n) \mid \ell(\lambda) \text{ odd}\}|=Q_{even, e-o}(n).$$
\end{proposition}

The statement of this proposition can also be interpreted as the excess in the number of distinct partitions in two colors, red and blue, with an even number of red parts over the number of distinct partitions in two colors with an odd number of red parts equals the excess in the number of partitions into an even number of distinct even parts over the number of partitions into an odd number of distinct even parts.

\begin{proof} Let $\mathcal E_{\mathcal{QQ}}(n)$ be the subset of $\mathcal{QQ}(n)$ defined by $$\mathcal E_{\mathcal{QQ}}(n)=\{(\lambda, \mu)\in \mathcal{QQ}(n) \mid \lambda= \mu\}.$$ We define a transformation  $\zeta$ on $\mathcal {QQ}(n)\setminus \mathcal E_{\mathcal{QQ}}(n)$ as follows. If $(\lambda, \mu)\in \mathcal{QQ}(n)$ with $\ell(\lambda)\not \equiv \ell(\mu)\bmod 2$, let $\zeta(\lambda,\mu)=\zeta(\mu, \lambda)$. If $\ell(\lambda) \equiv \ell(\mu)\bmod 2$ and $\lambda\neq \mu$,  let $i$ be  the smallest positive integer such that  $\lambda_i\neq \mu_i$. 

(i) If $\lambda_i>\mu_i$, let $\zeta(\lambda, \mu)= (\lambda\setminus (\lambda_i),  \mu\cup(\lambda_i))$. 

(ii) If $\lambda_i<\mu_i$, let $\zeta(\lambda, \mu)=(\lambda \cup (\mu_i),  \mu\setminus(\mu_i))$. 

Then, $\zeta$ is an involution on $\mathcal {QQ}(n)\setminus \mathcal E_{\mathcal{QQ}}(n)$ mapping pairs from case (i) to pairs from case (ii) and vice-versa. Moreover, $\zeta(\lambda, \mu)$ reverses the parity of the length of $\lambda$. Thus, \begin{align*} |\{(\lambda, \mu)& \in \mathcal{QQ}(n) \mid \ell(\lambda) \text{ even}\}|-|\{(\lambda, \mu)\in \mathcal{QQ}(n) \mid \ell(\lambda) \text{ odd}\}|\\ & =|\{(\lambda, \lambda)\in {\mathcal{QQ}}(n) \mid \ell(\lambda) \text{ even}\}|-|\{(\lambda, \lambda)\in {\mathcal{QQ}}(n) \mid \ell(\lambda) \text{ odd}\}|.\end{align*}
Mapping $(\lambda, \lambda)$ to $2\lambda$ completes the proof. 
\end{proof}
 For the rest of the article, whenever we refer to  the zeta transformation, we mean the transformation of Proposition \ref{zeta}. 

\begin{remark}If in the proof above, the zeta transformation is applied to pairs $(\lambda, \mu)$ with $\lambda, \mu \in \mathcal Q_{odd}$, we obtain a combinatorial proof for $$(q;q^2)_\infty(-q;q^2)_\infty=(q^2;q^4)_\infty.$$
\end{remark}

 To simplify explanations,  for the remainder of the article, when we  write that an involution is applied to a given set of partitions, it is implied that the exceptional set has been removed.

We end this section by noting that 
 using the involutions $\varphi_{K}$ and $\varphi_{A_1}$, we obtain a combinatorial proof of the first identity in Theorem \ref{ped_e-o}. 

To prove the second identity in Theorem \ref{ped_e-o}, i.e.,   $pod_{e-o}(n)=Q_{0,e-o}(n)$ for $n\geq 0$,  we  use Gupta's involution $\varphi_G$. First note that, if $\lambda=(\lambda^e, \lambda^o)$, then $\ell(\lambda)\equiv \ell (\lambda^e)+n \bmod 2$. 
Fix a partition $\lambda^o$ with distinct odd parts and size at most $n$. Consider the subset of $\POD(n)$ consisting of partitions whose odd parts are precisely the parts of $\lambda^o$, i.e., $$\POD_{\lambda^o}(n)=\{(\lambda^e, \lambda^o)\vdash n \mid \lambda^e \text{ has even parts}\}.$$ Then, $\lambda^e_{/2}$ is any partition in $\mathcal P((n-|\lambda^o|)/2)$.
The transformation $(\lambda^e,\lambda^o)\to(2\varphi_G(\lambda_{/2}^e), \lambda^o)$ is a sign reversing  involution on  $$\POD_{\lambda^o}(n)\setminus \{(\lambda^e, \lambda^o)\in \POD_{\lambda^o}(n) \mid \lambda^e \text{ distinct parts $\equiv 2 \bmod 4$}\}.$$ 
Observing that $\POD(n)=\bigcup_{\lambda^o\in \mathcal O} \POD_{\lambda^o}(n)$, it follows that $pod_{e-o}(n)=Q_{0,e-o}(n)$.

 \section{Combinatorial proof of Theorem \ref{FGK}}
 
 In this section we give two combinatorial proofs of Theorem \ref{FGK} which states that  for $n\geq 0$, $$\sum_{j=-\infty}^\infty (-1)^{j} ped(n-j(3j+1)/2)=\begin{cases} (-1)^k & \text{ if } n=2k(3k+1) \text{ for some } k \in \mathbb Z\\ 0&  \text{ otherwise. }  \end{cases}$$

\begin{proof}[First proof]

Let $P_{<4}(n)$ be the number of partitions of $n$ with parts occurring at most thrice.  In Section \ref{prelim} we showed combinatorially that $ped(n)=b_4(n)$. A well known generalization of Glaisher's bijection \cite{Gl} shows that $b_4(n)=P_{<4}(n)$. In fact, in \cite{FGK} the theorem is stated in terms of $P_{<4}(n)$.

Let $\mathcal{QP}_{<4} (n)=\{(\lambda, \mu)\vdash n \mid \lambda \in \mathcal P_{<4}, \mu \in \mathcal Q\}$. The transformation   $(\lambda, \mu)\to (\lambda, \varphi_F(\mu))$ on $\mathcal{QP}_{<4}(n)$ shows  that \begin{equation}\label{gfleft} \sum_{j=-\infty}^\infty (-1)^{j} ped(n-j(3j+1)/2)\end{equation} is the generating function for 
\begin{equation}\label{left} |\{(\lambda, \mu)\in \mathcal{QP}_{<4}  (n)\mid \ell(\mu) \text{ even}\}|-|\{(\lambda, \mu)\in \mathcal{QP}_{<4}  (n)\mid \ell(\mu) \text{ odd}\}|.\end{equation}

Let $(\lambda, \mu)\in \mathcal{QP}_{<4} (n)$ and define $\alpha=\lambda\cup \mu$. Then $\alpha$ has parts occurring at most four times. We examine the pairs $(\lambda, \mu)\in \mathcal{QP}_{<4}  (n)$ that result in the same partition $\alpha$. 
Let $\gamma$ be the partition into distinct parts that consists of one copy of each part of $\alpha$ that occurs exactly four times. Then all parts of $\gamma$ appear in $\alpha\setminus \gamma$ exactly three times. Let $\nu$ be the partition into distinct parts that consists of one copy of each part of $\alpha\setminus \gamma$ that is not a part of  $\gamma$.  For each subset $\beta$ of $\nu$, the pair $(\alpha\setminus(\gamma\cup\beta), \gamma\cup\beta)$ is in $\mathcal{QP}_{<4} (n)$ and results in $\alpha$. If $\nu\neq \emptyset$, it is easily seen that the number of subsets of $\nu$ with even cardinality is equal to the number of subsets of $\nu$ with odd cardinality (fix part $a$ of $\nu$ and map a subset $\beta$ of $\nu$ with $a \not \in \beta$ to $\beta\cup(a)$). Thus, if $\nu\neq \emptyset$, the set of pairs $(\alpha\setminus(\gamma\cup\beta), \gamma\cup\beta)$ with $\beta$ a subset of $\nu$ contributes zero to \eqref{left}. If $\nu=\emptyset$, then each part of $\alpha$ has multiplicity $4$. Then there is only one pair in $(\lambda, \mu)\in \mathcal{QP}_{<4}  (n)$ such that $\alpha=\lambda\cup \mu$, namely $(^3\mu, \mu)$, where  $^3\mu$ is the partition whose parts are the parts of $\mu$ each with multiplicity $3$. Thus \eqref{left} equals 
  $$|\{(^3\mu, \mu)\mid \mu \in \mathcal Q, \ell(\mu) \text{ even}\}|-|\{(^3\mu, \mu) \mid  \mu \in \mathcal Q, \ell(\mu) \text{ odd}\}|.$$
Finally, using the transformation $(^3\mu, \mu)\to (^3\varphi_F(\mu), \varphi_F(\mu))$ shows that the only partition that contributes to \eqref{left} is  $(^3\mu, \mu)$ with $\mu$ a pentagonal partition. This completes  the proof.
\end{proof}

\begin{proof}[Second proof] Let $\mathcal{QPED}(n)=\{(\lambda, \mu)\vdash n \mid \lambda \in \PED, \mu \in Q\}.$ The transformation $(\lambda, \mu) \to (\lambda, \varphi_F(\mu))$ shows that \eqref{gfleft} is the generating function for \begin{equation}\label{left1} |\{(\lambda, \mu)\in \mathcal{QPED}  (n)\mid \ell(\mu) \text{ even}\}|-|\{(\lambda, \mu)\in \mathcal{QPED}  (n)\mid \ell(\mu) \text{ odd}\}|.\end{equation}

Let $(\lambda, \mu)=(\lambda^e, \lambda^o, \mu^e, \mu^o)\in \mathcal{QPED}(n)$. The zeta  transformation applied   to $(\lambda^e, \mu^e)$ shows that  \eqref{left1} equals \begin{align*}|\{(\lambda^e, \lambda^o, \lambda^e, \mu^o) & \in \mathcal{QPED}(n) \mid \ell(\lambda^e)+\ell(\mu^o)\text{ even}\}|\\ & -|\{(\lambda^e, \lambda^o, \lambda^e, \mu^o)\in \mathcal{QPED}(n) \mid \ell(\lambda^e)+\ell(\mu^o)\text{ odd}\}|.\end{align*} Let $\eta=\lambda^e\cup \mu^o$,  identify $(\lambda^e, \lambda^o, \lambda^e, \mu^o)$ with $(\lambda^e, \lambda^o, \eta)$, and map $\lambda^o$ to $\nu=\varphi_{Gl}(\lambda^o)\in \mathcal Q$. Then  \eqref{left1} equals \begin{align*} |\{(\lambda^e, \nu,\eta)\vdash n & \mid \lambda^e \in \mathcal Q_{even}, \,  \nu, \eta \in\mathcal Q,  \ell(\eta)\text{ even}\}|\\ & - |\{(\lambda^e, \nu,\eta)\vdash n  \mid \lambda^e \in \mathcal Q_{even}, \,  \nu, \eta \in\mathcal Q,  \ell(\eta)\text{ odd}\}|.\end{align*} 
We apply the zeta  transformation   to $(\nu, \eta)$ and proceed as in Proposition \ref{zeta} to see that  \eqref{left1} equals \begin{align*} |\{(\lambda^e,2\eta)\vdash n & \mid \lambda^e, 2\eta\in \mathcal Q_{even}, \ell(\eta)\text{ even}\}|\\ & -|\{(\lambda^e, 2\eta)\vdash n  \mid \lambda^e, 2 \eta\in \mathcal Q_{even}, \ell(\eta)\text{ odd}\}|.\end{align*} Next, we apply $\zeta$ to $(\lambda^e, 2\eta)$ and proceed as in Proposition \ref{zeta} to see that \eqref{left1} equals \begin{align*}  |\{4\eta\vdash n  \mid \eta \in \mathcal Q,  \ell(\eta)\text{ even}\}| -|\{4\eta\vdash n  \mid  \eta \in \mathcal Q,  \ell(\eta)\text{ odd}\}|.\end{align*} Finally applying the transformation $4\eta \to 4\varphi_F(\eta)$ concludes the proof. 

\end{proof}

\section{Combinatorial proof of Theorem \ref{Merca 1.1}}

The goal of this section is to  prove combinatorially   that for $n\geq 0$, $$\sum_{j\geq 0} (-1)^{T_j} ped(n-T_j)=\begin{cases} 1 &  \text{ if } n=2T_k \text{ for some } k \geq 0,\\ 0 &  \text{ otherwise. }  \end{cases}$$

We begin by  establishing two helpful results.  First, we provide a combinatorial proof for \begin{equation}\label{JTP3} (-q;q)_\infty(-q;q)_\infty(q;q)_\infty=\sum_{n=0}^\infty q^{n(n+1)/2}.\end{equation} To state the combinatorial interpretation of \eqref{JTP3} let $\mathcal D_3(n)$ be the set of distinct partitions in three colors, i.e., $$\mathcal D_3(n)=\{(\lambda, \mu, \eta)\vdash n \mid \lambda, \mu, \eta \in \mathcal Q\}.$$ and define $$D'_{3, e-o}(n):=|\{(\lambda, \mu, \eta) \in \mathcal D_3(n) \mid \ell(\lambda) \text{ even}\}|-|\{(\lambda, \mu, \eta) \in \mathcal D_3(n) \mid \ell(\lambda) \text{ odd}\}|.$$

\begin{proposition} \label{L1} For $n\geq 0$ we have $$ D'_{3, e-o}(n)= \begin{cases}{1} & \text{ if $n=T_k$ for some } k\geq 0,\\ 0 & \text{ otherwise. }  \end{cases}$$
\end{proposition}

\begin{proof}   Let $\delta_k=(k, k-1, k-2, \ldots, 2,1)$ be the staircase partition of length $k$ and size $T_k$. 
The transformation that maps $(\lambda, \mu, \eta)\in \mathcal D_3(n)$ to  $(\lambda, \varphi_{BM}^{-1}(\mu, \eta), \delta_k)$ is a bijection between $\mathcal D_3(n)$ and  $\mathcal{QP}(n)$, where $$\mathcal{QP}(n)= \{(\lambda, \gamma, \delta_k)\vdash n \mid \lambda \in \mathcal Q, \gamma \in \mathcal P, k\in \mathbb Z_{\geq 0}\}.$$

Denote by $\mathcal E_1(n)$ the following disjoint union of subsets of $\mathcal{QP}(n)$ $$\mathcal E_1(n)=\bigcup_{m,k \geq 0} \{(\pi, \gamma, \delta_k)\vdash m \mid \pi \in \mathcal E_Q,  \gamma \in \mathcal P\}.$$ Note that $\pi$ is a pentagonal partition, $\mathcal{E}_\mathcal{Q}$ is the exceptional set for $\varphi_F$, and $\gamma$ is an unrestricted partition.

 The transformation that maps $(\lambda, \gamma, \delta_k)\in \mathcal{QP}(n)$ to $(\varphi_F(\lambda), \gamma, \delta_k)$ is an involution on $\mathcal{QP}(n)\setminus \mathcal E_1(n)$ that reverses the parity of $\ell(\lambda)$. Thus $$ D'_{3, e-o}(n)= E_{1,e-o}(n), $$
where $ E_{1,e-o}(n)=|\{(\pi, \gamma, \delta_k)\in \mathcal E_1(n) \mid \ell(\pi) \text{ even}\}|-|\{(\pi, \gamma, \delta_k)\in \mathcal E_1(n) \mid \ell(\pi) \text{ odd}\}|$.

\medskip

Let $\mathcal E_2(n)$ be the subset of $\mathcal E_1(n)$ defined by $$ \mathcal E_2(n)=\begin{cases} \{(\emptyset, \emptyset, \delta_k)\}, & \text{ if } n=T_k \text{ for some } k\geq 0, \\ \emptyset & \text{ otherwise.}\end{cases}$$

The transformation that maps $(\pi, \gamma, \delta_k)\in \mathcal E_1(n)$ to $(\varphi_{BZ}(\pi, \gamma), \delta_k)$ is an involution of $\mathcal E_1(n)\setminus \mathcal E_2(n)$ that reverses the parity of $\ell(\pi)$. Hence $$ D'_{3, e-o}(n)= |\mathcal E_2(n)|=\begin{cases} 1 & \text{ if $n=T_k$ for some } k\geq 0,\\ 0 & \text{ otherwise. } \end{cases} $$
\end{proof}

\begin{corollary}\label{CL1} For $n\geq 0$ we have $$ DE'_{3, e-o}(n)= \begin{cases}1 & \text{ if $n=2T_k$ for some } k\geq 0,\\ 0 & \text{ otherwise,}  \end{cases}$$ where \begin{align*}DE'_{3, e-o}(n)=|\{(\lambda, \mu, \nu)\vdash n&  \mid \lambda, \mu, \nu \in \mathcal Q_{even}, \ell(\lambda) \text{ even}\}|\\ & -|\{(\lambda, \mu, \nu)\vdash n \mid \lambda, \mu, \nu \in \mathcal Q_{even}, \ell(\lambda) \text{ odd}\}|.\end{align*}
\end{corollary} 

\begin{proof} The proof is obtained by  doubling each part in all partitions involved in the proof of Proposition \ref{L1}. \end{proof}  

The next result is analogous to Proposition \ref{zeta}. Its   combinatorial proof uses similar ideas to the zeta transformation but it is much more involved. Recall that $\mathcal O(n)$ is the set of odd partitions of $n$. 
Let $\mathcal{OO}(n)=\{(\lambda, \mu)\vdash n \mid \lambda, \mu \in \mathcal O\}$.

\begin{proposition} \label{prop} For $n\geq 0$ we have \begin{equation}\label{leftprop} |\{(\lambda, \mu)\in \mathcal{OO}(n) \mid \ell(\lambda) \text{ even}\}|-|\{(\lambda, \mu)\in \mathcal{OO}(n) \mid \ell(\lambda) \text{ odd}\}|=Q_{even}(n).\end{equation}
\end{proposition}

The statement of this proposition can also be interpreted as the excess in the number of odd partitions in two colors, red and blue, with an even number of red parts over the number of odd partitions in two colors with an odd number of red parts equals the  number of partitions into  distinct even parts. In terms of generating functions, we obtain a combinatorial proof for $$\frac{1}{(-q;q^2)_\infty(q;q^2)_\infty}=(-q^2;q^2)_\infty.$$

\begin{proof}
Let  $\mathcal{E}_{\mathcal{OO}}(n)$ be the  subset of $\mathcal{OO}(n)$ defined by   $$\mathcal{E}_{\mathcal{OO}}(n)=\{(\lambda, \mu)\in \mathcal{OO}(n) \mid \lambda =\emptyset, \text{ parts of $\mu$ have even multiplicity}\}.$$ We will  construct an involution $\psi$ on: $\mathcal{OO}(n) \setminus \mathcal{E}_{\mathcal{OO}}(n)$ that reverses the parity of $\ell(\lambda)$. 

If $(\lambda, \mu)\in \mathcal {OO}(n)$ with $\ell(\lambda) \not \equiv \ell(\mu)\bmod 2$, let $\psi(\lambda, \mu)=(\mu, \lambda)$. 

Next, suppose that $(\lambda, \mu)\in \mathcal {OO}(n)\setminus \mathcal{E}_{\mathcal{OO}}(n)$ with $\ell(\lambda)  \equiv \ell(\mu)\bmod 2$. We define the following subsets of $\mathcal{OO}(n)$: $$\mathcal{O}_{e,e}(n)=\{(\lambda, \mu)\in \mathcal{OO}(n) \mid \ell(\lambda), \ell(\mu)\text{ even}\}$$ and $$\mathcal{O}_{o,o}(n)=\{(\lambda, \mu)\in \mathcal{OO}(n) \mid \ell(\lambda),\ell(\mu) \text{ odd}\}.$$ Thus, $\mathcal E_{\mathcal{OO}}(n)\subseteq  \mathcal{O}_{e,e}(n)$ and $(\lambda, \mu)\in \mathcal{O}_{o,o}(n) \cup\left(\mathcal{O}_{e,e}(n)\setminus \mathcal E_{\mathcal{OO}}(n) \right)$.

Before we define $\psi(\lambda, \mu)$, we introduce some necessary notation.  We denote by $\lambda_s$, respectively $\mu_s$,  the smallest part of $\lambda$, respectively $\mu$. If $\lambda = \emptyset$ (or $\mu = \emptyset$), we define $\lambda_s = \infty$ (or $\mu_s = \infty$).  Let $\mu^<$ be the partition consisting of the parts of $\mu$ that are less than $\lambda_s$. Then $\mu \setminus \mu^<$ is the partition consisting of the parts of $\mu$ larger or equal to $\lambda_s$. 
Similarly, let $\lambda^<$ be the partition consisting of the parts of $\lambda$ that are less than $\mu_s$. 

We write $\mathcal{O}_{o,o}(n) = \mathcal{A}_o \sqcup \mathcal{B}_o \sqcup \mathcal{C}_o \sqcup \mathcal{D}_o$, where 
\begin{align*}
\mathcal{A}_o  = \{ (\lambda, \mu) \in  & \mathcal{O}_{o,o}(n) \ |  \  \lambda_s \geq \mu_s\},\\
\mathcal{B}_o  =  \{ (\lambda, \mu) \in & \mathcal{O}_{o,o}(n) \ |   \  \lambda_s < \mu_s, \ell(\lambda^<) \ \textrm{even}\},\\
\mathcal{C}_o  = \{ (\lambda, \mu) \in & \mathcal{O}_{o,o}(n) \ |   \ \lambda _s < \mu_s, \ell(\lambda^<) \ \textrm{odd},\\
&  \ \textrm{only the largest part of} \ \lambda \ \textrm{has odd multiplicity} \},\\
\mathcal{D}_o  = \{ (\lambda, \mu) \in & \mathcal{O}_{o,o}(n) \ |   \ \lambda _s < \mu_s, \ell(\lambda^<) \ \textrm{odd}, \\
& \ \textrm{some part other than the largest part of} \ \lambda \ \textrm{has odd multiplicity} \}.
 \end{align*}
 
 Similarly, we write  $\mathcal{O}_{e,e}(n)\setminus \mathcal{E}_{\mathcal{OO}}(n)= \mathcal{A}_e \sqcup \mathcal{B}_e \sqcup \mathcal{C}_e \sqcup \mathcal{D}_e$, where
 \begin{align*}
 \mathcal{A}_e & =  \{ (\lambda, \mu) \in \mathcal{O}_{e,e}(n) \ | \ \lambda \neq \emptyset, \lambda_s \leq \mu_s\},\\
\mathcal{B}_e & =  \{ (\lambda, \mu) \in \mathcal{O}_{e,e}(n) \ | \ \lambda \neq \emptyset, \lambda_s > \mu_s, \ell(\mu^<) \ \textrm{even}\}, \\
\mathcal{C}_e & = \{ (\emptyset, \mu) \in \mathcal{O}_{e,e}(n) \ | \ \textrm{not all multiplicities in $\mu$ are even} \},\\
\mathcal{D}_e & =  \{ (\lambda, \mu) \in \mathcal{O}_{e,e}(n) \ | \ \lambda \neq \emptyset, \lambda_s > \mu_s, \ell(\mu^<) \ \textrm{odd}\}.
\end{align*}

 We  now define bijections $\psi_{\mathcal{I}}:\mathcal I_o\to \mathcal I_e$, where $\mathcal{I}\in \{\mathcal{A}, \mathcal{B}, \mathcal{C}, \mathcal{D}\}$. Examples for each mapping are provided following the proof of Theorem \ref{Merca 1.1}.

Define $\psi_{\mathcal{A}}: \mathcal{A}_o \to \mathcal{A}_e$ by $\psi_{\mathcal{A}}(\lambda, \mu) = (\lambda \cup (\mu_s), \mu \setminus (\mu_s))$. Then the inverse, $\psi^{-1}_{\mathcal{A}}: \mathcal{A}_e \to \mathcal{A}_o$, is given by  $\psi^{-1}_{\mathcal{A}}(\lambda, \mu)=(\lambda \setminus (\lambda_s), \mu \cup (\lambda_s))$. 

Define $\psi_{\mathcal{B}}:  \mathcal{B}_o \to \mathcal{B}_e$ by $\psi_{\mathcal{B}}(\lambda, \mu) = (\lambda \setminus \lambda^< \cup (\mu_s), \mu \setminus (\mu_s) \cup \lambda^<)$. Note $\lambda \setminus \lambda^< \neq \emptyset$ because $\ell(\lambda)$ is odd but $\ell(\lambda^<)$ is even. The inverse  $\psi^{-1}_{\mathcal{B}}: \mathcal{B}_e \to \mathcal{B}_o$ is given by  $\psi^{-1}_{\mathcal{B}}(\lambda, \mu)=(\lambda \setminus (\lambda_s) \cup \mu^<, \mu \setminus \mu^< \cup (\lambda_s))$.

If $(\lambda, \mu) \in \mathcal{C}_e \sqcup \mathcal{D}_e$, then $\ell(\mu)$ is even and not all the multiplicities of parts in $\mu$ are even, so there must be at least two odd multiplicities. Let $\mu_i$ be the smallest part of $\mu$ such that the number of parts of $\mu$ smaller than $\mu_i$ is odd. Denote by $\mu^{<<}$ the partition whose parts are the parts of $\mu$ smaller than $\mu_i$. 
 Similarly, if $(\lambda, \mu) \in \mathcal{D}_o$, then there exists some part of $\lambda$ with odd multiplicity that is not the largest part, so we can define $\lambda_i$ and $\lambda^{<<}$. Note that by definition $\lambda^{<<} \neq \lambda$.

Define $\psi_{\mathcal{C}}:  \mathcal{C}_o \to \mathcal{C}_e$ by $\psi_{\mathcal{C}}(\lambda, \mu) = (\emptyset, \mu \cup \lambda)$. Note that only the largest part of $\lambda$ has odd multiplicity, so $\ell(\lambda^<)$ being odd implies that $\lambda^< = \lambda$, which means $\mu_s$ is larger than all the parts of $\lambda$. This also shows that $\mu \cup \lambda$ cannot have all parts with even multiplicity. The inverse $\psi^{-1}_{\mathcal{C}}: \mathcal{C}_e \to \mathcal{C}_o$ is given by  $\psi^{-1}_{\mathcal{C}}(\emptyset, \mu)=(\mu^{<<}, \mu \setminus \mu^{<<})$.  

Define $\psi_{\mathcal{D}}:  \mathcal{D}_o \to \mathcal{D}_e$ by $\psi_{\mathcal{D}}(\lambda, \mu) = (\lambda \setminus \lambda^{<<}, \mu \cup \lambda^{<<})$. Note that $\lambda_i \leq \mu_s$ as $\ell(\lambda^<)$ is odd and $\lambda_i$ is minimal. Recall, $\lambda \setminus \lambda^{<<} \neq \emptyset$ as $\lambda_i \in \lambda \setminus \lambda^{<<}$. The inverse $\psi^{-1}_{\mathcal{D}}: \mathcal{D}_e \to \mathcal{D}_o$ is given by  $\psi^{-1}_{\mathcal{D}}(\lambda, \mu)=(\lambda \cup \mu^{<<}, \mu \setminus \mu^{<<})$. 

Each of these mappings reverses the parity of  $\ell(\lambda)$. We define $\psi$ on $\mathcal{O}_{o,o}(n) \cup\left(\mathcal{O}_{e,e}(n)\setminus \mathcal E_{\mathcal{OO}}(n) \right)$ by the corresponding $\psi_{\mathcal I}$ or $\psi^{-1}_{\mathcal I}$, where $\mathcal{I}\in \{\mathcal{A}, \mathcal{B}, \mathcal{C}, \mathcal{D}\}$.

This shows that that the left hand side of \eqref{leftprop} equals $| \mathcal{E}_{\mathcal{OO}}(n)|$. To finish the proof, we  construct a bijection $\psi_\mathcal{E}: \mathcal{E}_{\mathcal{OO}}(n) \to \mathcal{Q}_{even}(n).$ 

Given $(\emptyset, \mu) \in \mathcal{E}_{\mathcal{OO}}(n)$, let $\widetilde\mu$ be the partition consisting of the different parts of $\mu$ each with half its multiplicity in $\mu$. We use Glaisher's transformation on $\widetilde\mu$ to obtain a partition with distinct parts. Then $2\varphi_{Gl}(\widetilde\mu) \in \mathcal{Q}_{even}(n)$  and  the transformation $(\emptyset, \mu)\to 2\varphi_{Gl}(\widetilde\mu)$ gives the desired bijection. 
\end{proof}

To complete the combinatorial proof of Theorem \ref{Merca 1.1}, we need to show combinatorially that \begin{equation}\label{deped} DE'_{3, e-o}(n)=\sum_{j\geq 0}(-1)^{T_j}ped(n-T_j).\end{equation}
 Denote by $\mathcal{PPED}(n)$ the set $$\mathcal{PPED}(n)=\{(\lambda, \mu)\vdash n  \mid \lambda, \mu \in \PED\}.$$ Using the transformation $\varphi_{A_1}$ on $\lambda$, the right hand side of \eqref{deped} is equal to \begin{align*}|\{(\lambda, \mu)\in \mathcal{PPED}(n)  \mid  \ell(\lambda) \text{ even}\}|-|\{(\lambda, \mu)\in \mathcal{PPED}(n)  \mid  \ell(\lambda) \text{ odd}\}|.\end{align*}
Given $(\lambda, \mu)\in \mathcal{PPED}(n)$, we write  $(\lambda, \mu)=(\lambda^e, \lambda^o, \mu^e, \mu^o)$ as usual.  Fix a partition $\lambda^e$  with distinct even parts and size at most $n$. Consider the subset of $\mathcal{PPED}(n)$ such that the even parts of the first partition are precisely the parts of $\lambda^e$, i.e., $$\mathcal{PPED}_{\lambda^e}(n)=\{(\lambda^e,  \lambda^o, \mu^e, \mu^o)\vdash n \mid \mu^e\in \mathcal Q_{even}, \lambda^o, \mu^o \in \mathcal O\}.$$
Then, applying the involution $\psi$ to $(\lambda^o, \mu^o)$ and proceeding as in the proof of Proposition \ref{prop}, we see that 
\begin{align*} |\{(\lambda^e, \lambda^o,\mu^e,  \mu^o)\in & \mathcal{PPED}_{\lambda^e}(n)\mid \ell(\lambda^o) \text{ even}\}|\\ &  - |\{(\lambda^e, \lambda^o, \mu^e,  \mu^o)\in \mathcal{PPED}_{\lambda^e}(n)\mid \ell(\lambda^o) \text{ odd}\}| \\ & \qquad  \qquad \qquad = |\{(\lambda^e, \mu^e, \gamma) \vdash n \mid \mu^e, \gamma\in \mathcal Q_{even}\}|.\end{align*}
Summing after all partitions $\lambda^e\in \mathcal Q_{even}$, shows that the right hand side of \eqref{deped} equals $DE'_{3, e-o}(n)$.

\medskip

We conclude this section with examples illustrating  the bijections $\psi_{\mathcal{A}}, \psi_{\mathcal{B}}, \psi_{\mathcal{C}},$ and $\psi_{\mathcal{D}}$. Below, marked boxes in one partition of the pair are moved to the other partition. 

\ytableausetup{smalltableaux} 

\begin{example}

Let $\lambda = (3, 1, 1), \mu = (5, 3, 1)$. Then $\mu_s \leq \lambda_s$ and $\psi_{\mathcal{A}}(\lambda, \mu) = (\delta, \gamma),$ where $\delta = (3, 1, 1, 1), \gamma = (5, 3)$. \\

\begin{center}
$(\lambda, \mu) =$ \Bigg( \ydiagram[*(white)]
{3,1,1}\, ,
\ydiagram[*(white)]
{5,3}
*[*(white)\bullet]{5,3,1} \Bigg) $\mapsto $
\Bigg( \ydiagram[*(white)]
{3,1,1}
*[*(white)\bullet]{3,1,1,1}\, ,
 \ydiagram[*(white)]
{5,3} \Bigg) $ = (\delta, \gamma)$
\end{center}

\end{example}

\begin{example}

Let $\lambda = (3, 1, 1), \mu = (5, 3, 3)$. Then $\mu_s > \lambda_s$ and $\ell(\lambda^<) = 2$ is even. Then $\psi_{\mathcal{B}}(\lambda, \mu) = (\delta, \gamma),$ where $\delta = (3, 3), \gamma = (5,3,1, 1)$.\\ 

\begin{center}
$(\lambda, \mu) =$ \Bigg( \ydiagram[*(white)]
{3,}
*[*(white)\bullet]{3,1,1}\, ,
\ydiagram[*(white)]
{5,3}
*[*(white)\bullet]{5,3,3} \Bigg) $\mapsto $
\Bigg( \ydiagram[*(white)]
{3}
*[*(white)\bullet]{3,3}\, ,
 \ydiagram[*(white)]
{5,3}
*[*(white)\bullet]{5,3,1,1} \Bigg) $ = (\delta, \gamma)$
\end{center}

\end{example}

\begin{example}

Let $\lambda = (3, 1, 1), \mu = (7,5,5)$. Then $\mu_s > \lambda_s$, $\ell(\lambda^<) = 3$ is odd (in fact, $\lambda^< = \lambda$), and the only part of $\lambda$ with odd multiplicity is the first part. Then $\psi_{\mathcal{C}}(\lambda, \mu) = (\delta, \gamma),$ where $\delta = \emptyset, \gamma = (7,5, 5,3,1, 1)$.\\

\begin{center}
$(\lambda, \mu) =$ \Bigg( \ydiagram[*(white)\bullet]
{3,1,1}\, ,
\ydiagram[*(white)]
{7,5,5} \Bigg) $\mapsto $
\Bigg( $\emptyset$ ,
 \ydiagram[*(white)]
{7,5,5}
*[*(white)\bullet]{7,5,5,3,1,1} \Bigg) $ = (\delta, \gamma)$
\end{center}

\end{example}

\begin{example}

Let $\lambda = (5, 5, 3, 1, 1), \mu = (7)$. Then $\mu_s > \lambda_s$, $\ell(\lambda^<) = 5$ is odd, and there is a part of $\lambda$ with odd multiplicity that is not the first part. Then $\lambda_i = 5$ and $\psi_{\mathcal{D}}(\lambda, \mu) = (\delta, \gamma),$ where $\delta = (5, 5), \gamma = (7,3,1, 1)$.\\

\begin{center}
$(\lambda, \mu) =$ \Bigg( \ydiagram[*(white)]
{5,5}
*[*(white)\bullet]{5,5,3,1,1}\, ,
\ydiagram[*(white)]
{7} \Bigg) $\mapsto $
\Bigg( \ydiagram[*(white)]
{5,5}\, ,
\ydiagram[*(white)]
{7}
*[*(white)\bullet]{7,3,1,1} \Bigg) $ = (\delta, \gamma)$
\end{center} 
\end{example}

For the remainder of the article, the transformation $\psi$ refers to the transformation of Proposition \ref{prop}.
\section{Combinatorial Proof of Theorem \ref{Merca 1.2}}

In this section we prove combinatorially that for $n\geq 0$, \begin{equation}\label{m2} \sum_{j=-\infty}^\infty (-1)^{j} ped(n-2j^2)=\begin{cases} 1&  \text{ if } n=T_k \text{ for some } k \geq 0,\\ 0&  \text{ otherwise. }  \end{cases}\end{equation}

Doubling parts in each overpartition involved in Andrews proof of  Gauss's second theta identity,  one obtains a combinatorial proof for $$\overline{pe}_{e-o}(n)= \begin{cases}2(-1)^{m} & \text{ if $n=2m^2$ for some } m\geq 0,\\ 1 & \text{ if $n=0$},
\\ 0 & \text{ otherwise, }  \end{cases}$$ where $\overline{pe}(n)$ is the number of overpartitions of $n$ with even parts. 
Thus, if we define $$\overline{\mathcal P} \PED(n)=\{(\overline \lambda, \mu)\vdash n \mid \overline\lambda \in \overline{\mathcal{PE}},  \  \mu \in \PED\},$$ 
to prove \eqref{m2}, we need to show that \begin{equation}\label{leftop} |\{(\overline \lambda, \mu)\in \overline{\mathcal P} \PED(n) \mid \ell(\overline \lambda) \text{ even}\}|-|\{(\overline \lambda, \mu)\in \overline{\mathcal P} \PED(n) \mid \ell(\overline \lambda) \text{ odd}\}|\end{equation} equals $1$ if $n$ is a triangular number and $0$ otherwise. 

As before, we write $\mu=(\mu^e, \mu^o),$ and we also write $\overline \lambda=( \alpha, \beta)$, where $ \alpha$ consists of the overlined parts of $\overline \lambda$ and $\beta$ consists of the nonoverlined parts of $\overline \lambda$. Then, we write $(\overline \lambda, \mu)\in \overline{\mathcal P} \PED(n)$ as $(\alpha, \beta, \mu^e, \mu^o)$. We note that the order on the partitions in the quadruple is important. In the quadruple, the parts of the first partition $\alpha$ are not overlined. However, when we create $(\overline \lambda, \mu)\in \overline{\mathcal P} \PED(n)$ from the quadruple, we overline the parts of $\alpha$. Thus, in $(\alpha, \beta, \mu^e, \mu^o)\in \overline{\mathcal P} \PED(n)$, partitions $\alpha$ and $\mu^e$ have parts even and distinct, $\beta$ is a partition with even parts (possibly repeated), and $\mu^o$ is a partition with odd parts (possibly repeated). 

We apply the zeta transformation to $(\alpha, \mu^e)$ and proceed as in Proposition \ref{zeta} to see that 
\begin{align*}| \{(\alpha, \beta, \mu^e, \mu^o)\in \overline{\mathcal P} \PED(n) & \mid \ell(\alpha)+\ell(\beta) \text{ even}\}|\\ & -|\{(\alpha, \beta, \mu^e, \mu^o)\in \overline{\mathcal P}  \PED(n) \mid \ell(\alpha)+\ell(\beta) \text{ odd}\}|\\   = | \{(2\alpha,  \beta, \mu^o) \vdash n  \mid  \alpha \in & \mathcal Q_{even},  \beta \text{ even parts}, \mu^o \in \mathcal O, \ell(\alpha)+\ell(\beta) \text{ even}\}|\\   -|  \{( 2\alpha,  \beta, \mu^o) \vdash & n   \mid \alpha \in \mathcal Q_{even},  \beta \text{ even parts}, \mu^o \in \mathcal O, \ell(\alpha)+\ell(\beta) \text{ odd}\}|.\end{align*} 
Next, we fix the partition $\alpha\in \mathcal Q_{even}$ and apply 
Gupta's transformation $\varphi_G$ to $\beta\cup \mu^o$ (which is an arbitrary partition of size $n-2|\alpha|$)  to see that 
 \begin{align*} |\{(2\alpha,  \beta, \mu^o) \vdash n &  \mid   \beta \text{ even parts}, \mu^o \in \mathcal O, \ell(\beta) \text{ even}\}|\\  &   -|  \{( 2\alpha,  \beta, \mu^o)  \vdash  n   \mid  \beta \text{ even parts}, \mu^o \in \mathcal O, \ell(\beta) \text{ odd}\}| \\ &   \ \ \ \ \ \ \ \ \ = |\{(2\alpha,  \nu)\vdash n \mid \nu \in\mathcal Q_{odd}\}|.\end{align*} 

Summing over all distinct partitions $\alpha$, we see that  \eqref{leftop} equals \begin{align*} |\{(2\alpha, \nu)\vdash n&  \mid \alpha\in \mathcal Q_{even}, \nu \in \mathcal  Q_{odd}, \ell(\alpha) \text{ even}\}|\\ & -|\{(2\alpha, \nu)\vdash n \mid \alpha\in \mathcal Q_{even}, \nu \in \mathcal Q_{odd}, \ell(\alpha) \text{ odd}\}|\end{align*}
 Since the number of odd parts in a partition of $n$ is congruent to $n \bmod 2$, it follows that  \eqref{leftop} equals $(-1)^n Q_{2,e-o}(n)$.
 
 Finally, the transformation $\varphi_K$  gives a combinatorial proof for $$(-1)^n Q_{2,e-o}(n) = \begin{cases}1 & \text{ if $n=T_i$ for some } i\geq 0,\\ 0 & \text{ otherwise.}  \end{cases}$$
 This concludes the proof.

\section{Combinatorial proof of Theorem \ref{pod-rec}($i$)}

In this section we prove combinatorially that for $n\geq 0$, 
\begin{equation}\label{podq0} Q_0(n) = pod(n)+2 \sum_{k=1}^{\infty} (-1)^{k}\, pod(n-4k^2).\end{equation} Recall that  $Q_0(n)$ is the number of distinct partitions of $n$ with parts $\not \equiv 0\pmod 4$. 

We begin as in the combinatorial proof of Theorem \ref{Merca 1.2}. 
Let $\overline{\mathcal P}_4$ be the set of overpartitions with all parts  $\equiv 0\bmod 4$ and define 
$$\overline{\mathcal P}_4 \POD(n):=\{(\overline \lambda, \mu)\vdash n \mid \overline\lambda \in \overline{\mathcal P}_4,  \  \mu \in \POD \}.$$

We use Andrews' transformation $\varphi_{A_2}$ for the proof of the second Gauss identity and quadruple the parts in all partitions in the proof to see that the right hand side of \eqref{podq0} equals 
 $$|\{(\overline \lambda, \mu)\in\overline{\mathcal P}_4 \POD(n) \mid \ell(\overline \lambda) \text{ even}\}|-|\{(\overline \lambda, \mu)\in \overline{\mathcal P}_4 \POD(n) \mid \ell(\overline \lambda) \text{ odd}\}|.$$ 
 
 As before, we write $\overline \lambda=(\alpha, \beta)$,  where $\alpha$ consists of of the overlined parts of $\overline \lambda$ and $\beta$ consists of of the nonoverlined parts of $\overline \lambda$. We write $\mu=(\mu^{e,2}, \mu^{e,4}, \mu^o)$, where $\mu^{e,2}$, respectively $\mu^{e,4}$, is the partition consisting of the parts of $\mu$ congruent to $2$, respectively $0$,  $\bmod \ 4$, and $\mu^o$ consists of the odd parts of $\mu$. 
 
We denote by $4\mathcal Q$, respectively $2\mathcal Q_{odd},$ the set of all partitions into distinct parts congruent to $0$, respectively $2$,  modulo $4$. We fix $\alpha \in 4\mathcal Q$  and consider the subset of $\overline{\mathcal P}_4 \POD(n)$ with the overlined parts of $\overline\lambda$ exactly the parts of $\alpha$, i.e.  $$\overline{\mathcal P}_4 \POD_\alpha(n) =\left\{(\alpha, \beta, \mu^{e,2}, \mu^{e,4}, \mu^o)\vdash n \, \Big| \begin{array}{l} \beta \text{ parts $\equiv 0 \bmod 4$},\\ \mu^{e,2/4} \text{ parts $\equiv 2/0 \bmod 4$}, \mu^o\in \mathcal Q_{odd} \end{array}\right\}. $$ 
The transformation $(\beta, \mu^{e,2})\to 2\varphi_G((\beta \cup \mu^{e,2})_{/2})$ shows that \begin{align*}|\{  (\alpha, \beta,  \mu^{e,2},   \mu^{e,4},  \mu^o)\in & \overline{\mathcal P}_4 \POD_\alpha(n)  \mid\ell(\beta) \text{ even}\}|\\  &  - |\{ (\alpha,  \beta,  \mu^{e,2},  \mu^{e,4}, \mu^o)\in \overline{\mathcal P}_4 \POD_\alpha(n) \mid\ell(\beta) \text{ odd}\}|\end{align*} \begin{align*} = |\{(\alpha,\gamma,  \mu^{e,4}, \mu^o)\vdash n \mid  \gamma  \in 2\mathcal Q_{odd}, \mu^{e,4} & \text{ parts  $\equiv 0\bmod 4$},  \mu^o\in \mathcal Q_{odd}\}|\\ & =:\overline{\mathcal P}_4 \POD'_\alpha(n). \end{align*}

Summing over all $\alpha$,  shows that the right hand side of \eqref{podq0} equals 
 \begin{align*} |\{(\alpha,\gamma,  \mu^{e,4},  \mu^o)& \in \overline{\mathcal P}_4 \POD'(n) \mid \ell(\alpha) \text{ even}\}|\\ & -|\{(\alpha,\gamma,  \mu^{e,4}, \mu^o)\in \overline{\mathcal P}_4 \POD'(n) \mid \ell(\alpha) \text{ odd}\}|,\end{align*} where $\overline{\mathcal P}_4 \POD'(n) =\bigcup_{\alpha\in 4\mathcal Q}\overline{\mathcal P}_4 \POD'_\alpha(n)$. 
 
 Next, we map $(\alpha,\gamma,  \mu^{e,4}, \mu^o)$ to $(4\varphi_F(\alpha_{/4}),\gamma,  \mu^{e,4}, \mu^o)$ to see that the right hand side of \eqref{podq0} equals 
 \begin{align*} |\{(4\pi,\gamma,  \mu^{e,4},  \mu^o)& \in \overline{\mathcal P}_4 \POD'(n) \mid \ell(\pi) \text{ even}\}|\\ & -|\{(4\pi,\gamma,  \mu^{e,4}, \mu^o)\in \overline{\mathcal P}_4 \POD'(n) \mid \ell(\pi) \text{ odd}\}|,\end{align*} where $\pi$ is a pentagonal partition. 
 Finally, applying $\varphi_{BZ}$ to $(\pi, \mu^{e,4}_{/4})$ we see that the right hand side of \eqref{podq0} equals $$ |\{( \gamma,  \mu^o)\vdash n \mid \gamma \in 2\mathcal Q_{odd},  \mu^o\in \mathcal Q_{odd}\}|=Q_0(n).$$

 \section{Combinatorial proofs  of Theorems \ref{Merca 4.1}, \ref{Merca 4.3}, \ref{Merca 5.1}, and \ref{C 5.4}}

In this section, we provide combinatorial proofs for the remaining theorems in \cite{M17}. 

\subsection{Proof of Theorem \ref{Merca 4.1}}

We prove combinatorially that for $n\geq 0$ we have $$ped(n)=\sum_{k=0}^\infty pod(n-2T_k).$$

\begin{remark}In \cite{M17}, the theorem  is stated with $pod(n-2T_k)$ replaced by $p_2(n-2T_k)$. However, in Section \ref{prelim}, we gave the bijection between $\mathcal{POD}(n)$ and $\mathcal P_2(n)$.
\end{remark}

We create a bijection $h: \PED(n)\to \bigcup_{k=0}^\infty \POD(n-2T_k)$. 

Start with $\lambda=(\lambda^e, \lambda^o)\in \mathcal{PED}(n)$. Let $\eta^o$ be the partition into odd distinct parts whose parts are precisely the parts of $\lambda^o$ that have odd multiplicity. Then each part in $\lambda^o\setminus \eta^o$ is odd and has even multiplicity. We transform $\lambda^o\setminus \eta^o$ into a partition $\mu^e$ with distinct even parts as follows. Let $\nu^o$ be the partition whose parts are the different parts of $\lambda^o\setminus \eta^o$ each occurring with half its multiplicity in $\lambda^o\setminus \eta^o$. Then $\mu^e$ is defined as $2\varphi_{Gl}(\nu^o)$ and clearly $|\lambda^e|+|\eta^o|+|\mu^e|=n$. 
Next we define $\eta^e$ to be the partition $2\varphi^{-1}_{BM}(\lambda^e_{/2}, \mu^e_{/2})$. Thus, $\eta^e$ is a partition into even parts of size $|\lambda^e|+|\mu^e|-2T_k$ for some $k\geq 0$. 
 We define $h(\lambda) =(\eta^e, \eta^o)\in \POD(n-2T_k)$. 
 
 To define the inverse of $h$, start with $(\eta^e, \eta^o)\in \POD(n-2T_k)$ for some $k\geq 0$. Then $2\varphi_{BM}(\eta^e_{/2})$ is a partition of $n-|\eta^o|$ with distinct even parts in two colors.  We write this partition as $(\lambda^e, \alpha)$. Let $\beta$ be the partition obtained by doubling the multiplicity of each part of $\varphi_{Gl}^{-1}(\alpha_{/2})$.  Hence, $\beta$ has odd parts each with even multiplicity and we let $\lambda^o=\eta^o\cup\beta$. Then $h^{-1}(\eta^e, \eta^o)=(\lambda^e, \lambda^o)\in \PED(n)$. 

\begin{remark} It follows from Theorem \ref{Merca 4.1}  that for $n\geq 0$ we have   $$ped(n)-pod(n)=\sum_{k=1}^\infty pod(n-2T_k).$$ This gives another proof that $ped(n)-pod(n) \geq 0$. Moreover, from the discussion in Section \ref{prelim}, it follows that $$\sum_{k=1}^\infty pod(n-2T_k)=\left|\left\{\lambda\in \PED(n) \mid m(1)< \displaystyle \sum_{a\in \lambda^e}2^{val_2(a)}\right\}\right|.$$
\end{remark}

\subsection{Proof of Theorem \ref{Merca 4.3}}

 We prove combinatorially that for $n\geq 0$ we have  \begin{equation}\label{4.3}\sum_{j=0}^\infty (-1)^{T_j}pod(n-T_j)=\begin{cases}1 & \text{ if } n=0,\\ 0 & \text{ if } n>0. \end{cases}\end{equation}

The statement is clearly true if $n=0$ since $\mathcal{POD}(0)=\{\emptyset\}$. Let $n>0$, and define $$\mathcal{Q}_2\mathcal{POD}(n)=\{(\eta, \lambda) \vdash n \mid \eta \in \mathcal Q_2, \lambda \in \mathcal{POD} \}.$$ Using $\varphi_K$, shows that proving \eqref{4.3} is equivalent to showing that \begin{equation}\label{Q2POD}|\{(\eta, \lambda)\in \mathcal{Q}_2\mathcal{POD}(n) \mid \ell(\eta) \text{ even}\}|-|\{(\eta, \lambda)\in \mathcal{Q}_2\mathcal{POD}(n) \mid \ell(\eta) \text{ odd}\}|\end{equation} equals $0$. We write $(\eta, \lambda)\in \mathcal{Q}_2\mathcal{POD}(n)$ as $(\eta^e, \eta^o, \lambda^e\, \lambda^o)$. Thus  $\eta^o,\lambda^o$ have distinct odd parts, $\eta^e$ has distinct  parts $\equiv 0\bmod 4$, and $\lambda^e$ has  even parts. We apply the zeta transformation to $(\eta^o, \lambda^o)$ and  proceed as in Proposition \ref{zeta} to see that  \eqref{Q2POD} is equal to 
\begin{align*}
|\{(\eta^e, 2\eta^o, \lambda^e) \vdash n \mid \eta^o \in \mathcal{Q}_{odd}, \eta^e \in 4\mathcal Q, \lambda^e \ \textrm{even parts}, \ell(\eta^o) + \ell(\eta^e) \ \textrm{even}\}|\\
- |\{(\eta^e, 2\eta^o, \lambda^e) \vdash n \mid \eta^o \in \mathcal{Q}_{odd}, \eta^e \in 4\mathcal{Q}, \lambda^e \ \textrm{even parts}, \ell(\eta^o) + \ell(\eta^e) \ \textrm{odd}\}|.
\end{align*}

We map the triple $(\eta^e, 2\eta^o, \lambda^e)$ to the overpartition $(\alpha, \beta)$ of $n/2$ whose overlined parts are the parts of $\alpha=\eta^o\cup \eta^e_{/2}$ and the non-ovelined parts are the parts of $\beta=\lambda^e_{/2}$.  It is easy to see that this transformation is invertible. Thus,   \eqref{Q2POD} equals
$$|\{(\alpha, \beta) \in \overline{\mathcal{P}}(n/2) \mid \ell(\alpha) \ \textrm{even}\}| - |\{(\alpha, \beta) \in \overline{\mathcal{P}}(n/2) \mid \ell(\alpha) \ \textrm{odd}\}|.$$

It remains to show that if $n$ is even, the number of overpartitions of $n/2$ with an even number of overlined parts is equal to the number of overpartitions of $n/2$ with an odd number of overlined parts. As remarked in Section \ref{inv}, Andrews' involution $\varphi_{A_2}$ on $\overline{\mathcal P}(n)\setminus \mathcal E_{\overline{\mathcal P}(n)}$  reverses the parity of the number of overlined parts. Moreover, the set $ \mathcal{E}_{\overline{\mathcal P}(n)}$ is either empty or consists of two overpartitions: one with no overlined parts and one with one overlined part. This concludes the proof.

\subsection{Proof of Theorem \ref{Merca 5.1}}
 We prove combinatorially that for $n\geq 0$ we have $$ped(n)=\sum_{k\geq 0}\overline p\left(\frac{n}{2}-\frac{T_k}{2}\right).$$  
 
 We will  use  the bijective proof of Fu and Tang  \cite{FT20} for Watson's identity $$|\mathcal{Q}(n)|=\sum_{k\geq 0} p\left(\frac{n-T_k}{2}\right).$$ We denote their bijection by $\varphi_{FT}$. 

Let $n\geq 0$. We create a bijection between $\PED(n)$ and $\bigcup_{k=0}^\infty \overline{\mathcal P}\left(\frac{n}{2}-\frac{T_k}{2}\right)$. 

 Start with $\lambda= (\lambda^e, \lambda^o) \in \PED(n)$. Let $\alpha=\lambda^e_{/2}$ and $\beta=\varphi_{FT}(\varphi_{Gl}(\lambda^o))$. Thus, $\beta \in \mathcal P((|\lambda^o|-T_k)/2)$ for some $k\geq 0$. We map $\lambda$ to the overpartition whose overlined parts are the parts of $\alpha$ and the non-overlined parts are the parts of $\beta$. Then $(\alpha, \beta)\in \overline{\mathcal P}\left(\frac{n}{2}-\frac{T_k}{2}\right)$. The transformation is clearly invertible as we used bijections to define it.

\subsection{Proof of Theorem \ref{C 5.4}}

 We prove combinatorially that for $n\geq 0$ we have \begin{equation}\label{cor_5_4} \sum_{j=0}^\infty o_{e-o}(n-T_j)=\begin{cases} (-1)^k,  & \text{ if } n=2k(3k+1), \  k \in \mathbb Z,\\ 0,& \text{ otherwise. }\end{cases}\end{equation}

  First notice that if $\lambda \in \mathcal{O}(n)$, then $\ell(\lambda)\equiv n \bmod 2$, so $o_{e-o}(n)=(-1)^no(n)$.  Thus, if we define $S(n)=\sum_{j=0}^\infty(-1)^{T_j} o(n-T_j)$, the left hand side of \eqref{cor_5_4} equals $(-1)^nS(n)$. 
  Let $$\mathcal{OPED}(n)=\{(\lambda, \mu)\vdash n \mid \lambda \in \mathcal{PED}, \mu \in \mathcal{O}\}.$$ As in the proof of Theorem \ref{Merca 1.1}, $$S(n)=|\{(\lambda, \mu)\in \mathcal{OPED}(n) \mid \ell(\lambda) \text{ even}\}|-|\{(\lambda, \mu)\in \mathcal{OPED}(n) \mid \ell(\lambda) \text{ odd}\}|.$$ 
  We write $(\lambda, \mu)=(\lambda^e, \lambda^o, \mu)$, and apply the transformation $\psi$  to $(\lambda^o, \mu)$. Then, Proposition \ref{prop} implies \begin{align*} S(n)& = |\{(\lambda^e, \eta)\vdash n \mid \lambda^e, \eta \in \mathcal Q_{even}, \ell(\lambda^e) \text{ even}\}|\\ & \ \ \ \ \ \ -  |\{(\lambda^e, \eta)\vdash n \mid \lambda^e, \eta \in \mathcal Q_{even}, \ell(\lambda^e) \text{ odd}\}|.\end{align*} 
  By Proposition \ref{zeta}, \begin{align*}S(n)& =|\{\nu\in 4\mathcal Q \mid \ell(\nu) \text{ even}\}|-|\{\nu \in 4\mathcal Q \mid  \ell(\nu) \text{ odd}\}|.\end{align*} Finally, mapping $\nu$ to $4\varphi_F(\nu_{/4})$ gives  $$S(n)=\begin{cases} (-1)^k  & \text{ if } n=2k(3k+1), \, k \in \mathbb Z,\\ 0 & \text{ otherwise. }\end{cases}$$ Observing that $S(n)=0$ if $n$ is odd completes the proof.

\section{Beck-Type Identities}

Beck-type identities are companions to partition identities. If the original identity states that the number of partitions of $n$ satisfying condition $X$ equals the number of partitions of $n$ satisfying condition $Y$, then the Beck-type companion  identity  gives a combinatorial interpretation for the excess in the number of parts in all partitions of $n$ subject to $X$ over the number of parts  in all partitions of $n$ subject to $Y$. The first such identity, a companion for Euler's identity,  was conjectured by  Beck \cite{B1} and proved by Andrews \cite{A-Beck}.

In this  section we prove Beck-type companion identities to  the identities given by \eqref{4-reg} and \eqref{pod-mod 4}, i.e., for all $n\geq 0$ we have $$ped(n)=b_4(n)$$ and $$pod(n)=p_2(n).$$ 

Before we state and prove the Beck-type identities, we recall a useful generalization of the original work in \cite{A-Beck}. 

Let $S_1$ and $S_2$ be subsets of the positive integers. Denote by $\mathcal{O}_r(n)$ the set of partitions of $n$ with parts from $S_2$ and by  $\mathcal{D}_r(n)$ the set of partitions of $n$ with parts from $S_1$ and no part repeated $r$ or more times. If $|\mathcal{O}_r(n)| = |\mathcal{D}_r(n)|$ for all $n\geq 0$, the pair $(S_1, S_2)$ is called an Euler pair of order $r$. Subbarao \cite{S71} proved that  $(S_1, S_2)$ is  an Euler pair of order $r$ if and only if $rS_1 \subseteq S_1$ and $S_2 =S_1 \setminus rS_1$.
Ballantine and Welch \cite{BW21} proved the following theorem analytically and combinatorially, which gives a Beck-type identity for Euler pairs.
\begin{theorem} \label{bw} Suppose $(S_1, S_2)$ is  an Euler pair of order $r$ and let $n\geq 0$. 
The excess in the number of parts in all partitions of $\mathcal{O}_r(n)$ over the number of parts in all partitions of $\mathcal{D}_r(n)$ is equal to  the number of partitions with parts from $S_1$ and exactly one part repeated at least $r$ times. The excess also equals   the number of partitions with exactly one part (possibly repeated) from $rS_1$ and all other parts from $S_2$. 
\end{theorem}

In particular,  if $S_1$ is the set of even positive integers and $S_2$ is the set of positive integers congruent to $2 \bmod 4$, then $(S_1, S_2)$ is an Euler pair of order $2$. Thus, if $\mathcal B_{4, even}(n)$ is the set of 4-regular partitions of $n$ with even parts, then for $n\geq 0$ we have \begin{equation}\label{eq}
|\mathcal{B}_{4, even}(n)| = |\mathcal{Q}_{even}(n)|,
\end{equation} and we obtain the following corollary to Theorem \ref{bw}. 

\begin{corollary}\label{Beck1lem}
Let $n\geq 0$. The excess in the number of parts in all partitions in $\mathcal{B}_{4, even}(n)$ over the number of parts in all partitions in $\mathcal Q_{even}(n)$  is equal to  the number of partitions with even parts and exactly one part repeated. The excess also equals   the number of partitions with exactly one part (possibly repeated) congruent to $0 \bmod 4$ and all other parts congruent to $2 \bmod 4$. 
\end{corollary}

\subsection{A Beck-type companion identity to $ped(n)=b_4(n)$} Before we state the theorem, we introduce some notation. 
Let $b(n)$ be the excess in the number of parts in all partitions in $\mathcal B_4(n)$  over the number of parts in all partitions in $\PED(n)$. Let $\PED_1(n)$ be the set of partitions of $n$ with  one even part repeated, all other even parts distinct, and odd parts unrestricted. Let $\mathcal B_{4,1}(n)$ be the set of partitions of $n$ with exactly one part (possibly repeated)  $\equiv 0 \bmod 4$ and all other parts $\not \equiv 0 \bmod 4$. 

\begin{theorem}\label{Beck1}
For all $n \geq 0$, $b(n) = |\PED_1(n)| = |\mathcal B_{4,1}(n)|$. 
\end{theorem} 

Using Corollary \ref{Beck1lem} and its proofs in \cite{BW21}, we obtain both  analytic and combinatorial proofs of the theorem.

\begin{proof}[Analytic Proof]
Let $z, q\in \mathbb C$, $|z|, |q|<1$  (so that all series converge absolutely). Let $b_4(m,n)$, respectively $ped(m,n)$, be the number of partitions in $\mathcal B_4(n)$, respectively $\PED(n)$, with $m$ parts. We define
\begin{align*}
F(z, q) & := \sum_{m \geq 0} \sum_{n \geq0} b_4(m,n) z^m q^n,\\
G(z, q) & := \sum_{m \geq 0} \sum_{n \geq0} ped(m,n) z^m q^n.
\end{align*}
We have 
\begin{align*} 
F(z, q) & = \frac{1}{(zq^2; q^4)_\infty (zq; q^2)_\infty},\\
G(z, q) & = \frac{(-zq^2; q^2)_\infty}{(zq; q^2)_\infty}.
\end{align*}
Then,
\begin{align*}
\sum_{n=0}^{\infty}b(n)q^n &= \frac{\partial}{\partial z} \Bigg|_{z=1} (F(z, q) - G(z,q))\\
&= \frac{\partial}{\partial z} \Bigg|_{z=1} \frac{1}{(zq; q^2)_\infty}\Bigg(\frac{1}{(zq^2; q^4)_\infty} - (-zq^2; q^2)_\infty\Bigg)\\
& = \frac{\partial}{\partial z} \Bigg[\frac{1}{(zq; q^2)_\infty}\Bigg]_{z=1} \cdot \Bigg(\frac{1}{(q^2; q^4)_\infty} - (-q^2; q^2)_\infty\Bigg)\\
&+ \frac{1}{(q; q^2)_\infty}  \cdot \frac{\partial}{\partial z} \Bigg[\frac{1}{(zq^2; q^4)_\infty} - (-zq^2; q^2)_\infty \Bigg]_{z=1}\\
& = \frac{1}{(q; q^2)_\infty} \cdot \frac{\partial}{\partial z} \Bigg[\frac{1}{(zq^2; q^4)_\infty} - (-zq^2; q^2)_\infty \Bigg]_{z=1},
\end{align*}

\noindent where we invoked \eqref{eq} in generating functions form:
$$\displaystyle \frac{1}{(q^2; q^4)_\infty} = (-q^2, q^2)_\infty.$$

To complete the proof, we use the analytic proof of Corollary \ref{Beck1lem} which can be expressed as
\begin{align*}
\frac{\partial}{\partial z} \Bigg|_{z=1} \left(\frac{1}{(zq^2; q^4)_\infty} - (-zq^2; q^2)_\infty\right)& =  \frac{1}{(q^2; q^4)_\infty} \sum_{i = 1}^\infty\frac{q^{4i}}{1 - q^{4i}}\\
&=  (-q^2; q^2)_\infty \sum_{i = 1}^\infty\frac{q^{2(2i)}}{1 - q^{2(2i)}}.
\end{align*} 
Since $$
  \sum_{n=0}^{\infty}|\PED_1(n)|q^n = \frac{(-q^2; q^2)_\infty}{(q; q^2)_\infty} \cdot \sum_{i = 1}^\infty \frac{q^{2(2i)}}{1 - q^{2(2i)}}$$ and  
$$ \sum_{n=0}^{\infty}|\mathcal B_{4,1}(n)|q^n = \frac{1}{(q; q^2)_\infty (q^2; q^4)_\infty} \cdot \sum_{i = 1}^\infty \frac{q^{4i}}{1 - q^{4i}},
$$
it follows that 
 \begin{align*}
 \sum_{n=0}^{\infty}b(n)q^n & = \sum_{n=0}^{\infty}|\mathcal B_{4,1}(n)|q^n=\sum_{n=0}^{\infty}|\PED_1(n)|q^n. \end{align*}
\end{proof}

\begin{proof}[Combinatorial Proof]
As explained in Section \ref{prelim},  the transformation that maps $\lambda=(\lambda^e, \lambda^o)\in \mathcal B_4(n)$ to $(2\varphi_{Gl}(\lambda^e_{/2}), \lambda^o)\in \PED(n)$ is a bijection. Moreover, this bijection preserves the number of odd parts of $\lambda$. Thus, by Corollary  \ref{Beck1lem}, for a fixed partition $\lambda^o\in \mathcal O$, the number of parts in all partitions in $\{(\lambda^e, \lambda^o)\vdash n \mid \lambda^e \in \mathcal B_{4,even}\}$ minus the number of parts in all partitions in $\{(\lambda^e, \lambda^o)\vdash n \mid \lambda^e \in \mathcal Q_{even}\}$ equals the the number of partitions of $n-|\lambda^o|$ with even parts and exactly one part repeated. This excess also equals   the number of partitions of $n-|\lambda^o|$ with exactly one part (possibly repeated) congruent to $0 \bmod 4$ and all other parts congruent to $2 \bmod 4$. We map each  partition $\mu$ in the sets described above to $\mu\cup \lambda^o$ to obtain a partition of $n$ in $\PED_1(n)$, respectively $\mathcal B_{4,1}(n)$.  Considering  all $\lambda^o\in \mathcal O$  completes the proof. 
\end{proof}

\subsection{A Beck-type companion identity to $pod(n)=p_2(n)$} 
Let $b'(n)$ be the excess in the number of parts in all partitions in $\mathcal P_2(n)$ over the number of parts in all partitions in $\POD(n)$.  Let $\POD_1(n)$ be the set of partitions of $n$ with  one odd part repeated, all other odd parts distinct, and even parts unrestricted. Let $\mathcal P_{2,1}(n)$ be the set of partitions of $n$ with exactly one part (possibly repeated)  $\equiv 2 \bmod 4$ and all other parts $\not \equiv 2 \bmod 4$.

\begin{theorem}
For all $n \geq 1$, $ b'(n) = |\POD_1(n)|=|\mathcal P_{2,1}(n)|$. 
\end{theorem} 

Unlike the proof of Theorem \ref{Beck1}, here we cannot use the work of \cite{BW21}. There is no Euler pair of order $2$ with $S_1$ being the set of positive odd integers since $2S_1 \nsubseteq S_1$.  

\begin{proof}[Analytic Proof] Let $z, q\in \mathbb C$, $|z|, |q|<1$  (so that all series converge absolutely). Let $p_2(m,n)$, respectively $pod(m,n)$, be the number of partitions in $\mathcal P_2(n)$, respectively $\POD(n)$, with $m$ parts. 
 We define
\begin{align*}
F(z, q) & := \sum_{m \geq 0} \sum_{n \geq0} p_2(m,n) z^m q^n,\\
G(z, q) & := \sum_{m \geq 0} \sum_{n \geq0} pod(m,n) z^m q^n.
\end{align*}
We have
\begin{align*} 
F(z, q) & = \frac{1}{(zq^4; q^4)_\infty (zq; q^2)_\infty},\\
G(z, q) & = \frac{(-zq; q^2)_\infty}{(zq^2; q^2)_\infty}.
\end{align*}
Then,
\begin{align*}
\sum_{n=0}^{\infty}b'(n)q^n = \frac{\partial}{\partial z} \Bigg|_{z=1} (F(z, q) - G(z,q)).
\end{align*}

Using logarithmic differentiation,

\begin{align*}
 \frac{\partial}{\partial z} \Bigg|_{z=1} F(z, q)  & = \frac{1}{(q^4; q^4)_\infty (q; q^2)_\infty} \cdot \sum_{i = 1}^{\infty} \Bigg( \frac{q^{4i}}{1 - q^{4i}} +  \frac{q^{2i-1}}{1 - q^{2i-1}}  \Bigg),
 \end{align*}
 
 and 
 
 \begin{align*}
 \frac{\partial}{\partial z} \Bigg|_{z=1} G(z, q)  = \frac{(-q; q^2)_\infty}{(q^2; q^2)_\infty} \cdot \sum_{i = 1}^{\infty} \Bigg( \frac{q^{2i-1}}{1 + q^{2i-1}} + \frac{q^{2i}}{1 - q^{2i}} \Bigg).
 \end{align*}
 
 Using \eqref{pod-mod 4}, it follows that 
 
 \begin{align*}
 \sum_{n=0}^{\infty}&b(n)q^n  =  \frac{1}{(q^4; q^4)_\infty  (q; q^2)_\infty}\\
 & \cdot \Bigg[\sum_{i = 1}^{\infty} \Bigg(\frac{q^{4i}}{1 - q^{4i}} +  \frac{q^{2i-1}}{1 - q^{2i-1}} \Bigg) - \sum_{i = 1}^{\infty} \Bigg( \frac{q^{2i-1}}{1 + q^{2i-1}} + \frac{q^{2i}}{1 - q^{2i}} \Bigg)\Bigg].
 \end{align*}
 
\medskip 

  \noindent We  will show that
 \begin{align*}
 \sum_{i = 1}^{\infty} \Bigg(\frac{q^{4i}}{1 - q^{4i}} +  \frac{q^{2i-1}}{1 - q^{2i-1}} \Bigg) &- \sum_{i = 1}^{\infty} \Bigg( \frac{q^{2i-1}}{1 + q^{2i-1}} + \frac{q^{2i}}{1 - q^{2i}} \Bigg)\\
 & = \sum_{i = 1}^{\infty} \frac{q^{4i - 2}}{1 - q^{4i - 2}}.
 \end{align*}
  First, note that
\begin{align*}
\sum_{i = 1}^\infty \sum_{m = 1}^\infty q^{m(2i)} &= \sum_{i = 1}^\infty \sum_{\substack{m \geq 1 \\ odd}} q^{m(2i)} + \sum_{i = 1}^\infty \sum_{\substack{m \geq 1 \\ even}} q^{m(2i)}\\
&= \sum_{i = 1}^\infty  \sum_{\substack{m \geq 1 \\ even}}  q^{m(2i-1)} + \sum_{i = 1}^\infty \sum_{m =1}^\infty q^{m(4i)}.
\end{align*}
 Then
 \begin{align*}
& \sum_{i = 1}^{\infty} \Bigg( \frac{q^{4i}}{1 - q^{4i}} + \frac{q^{2i-1}}{1 - q^{2i-1}} \Bigg) - \sum_{i = 1}^{\infty} \Bigg( \frac{q^{2i-1}}{1 + q^{2i-1}} + \frac{q^{2i}}{1 - q^{2i}} \Bigg)\\
& = \sum_{i = 1}^\infty \sum_{m=1}^\infty \Big( q^{m(4i)} + q^{m(2i-1)} \Big) - \Bigg( \sum_{i = 1}^\infty \frac{q^{2i-1}}{1 + q^{2i-1}} + \sum_{i = 1}^\infty \sum_{m = 1}^\infty q^{m(2i)} \Bigg)\\
& = \sum_{i = 1}^\infty  \sum_{\substack{m \geq 1 \\ odd}} q^{m(2i-1)}-  \sum_{i = 1}^\infty \frac{q^{2i-1}}{1 + q^{2i-1}}\\
& = \sum_{i = 1}^\infty  \sum_{\substack{m \geq 1 \\ odd}} q^{m(2i-1)} -  \Bigg( \sum_{i = 1}^\infty  \sum_{\substack{m \geq 1 \\ odd}} q^{m(2i-1)} - \sum_{i = 1}^\infty  \sum_{\substack{m \geq 1 \\ even}} q^{m(2i-1)}\Bigg)\\
& = \sum_{i = 1}^\infty  \sum_{\substack{m \geq 1 \\ even}} q^{m(2i-1)}
 = \sum_{i = 1}^\infty \frac{q^{4i - 2}}{1 - q^{4i - 2}}.
\end{align*}
Since $$ \sum_{n=0}^\infty|\mathcal P_{2,1}(n)|q^n =  \frac{1}{(q^4; q^4)_\infty (q; q^2)_\infty} \cdot \sum_{i = 1}^\infty \frac{q^{4i - 2}}{1 - q^{4i - 2}}$$ and  $$ \sum_{n=0}^{\infty}|\POD_1(n)|q^n  = \frac{(-q; q^2)_\infty}{(q^2; q^2)_\infty}\cdot \sum_{i = 1}^\infty \frac{q^{2(2i - 1)}}{1 - q^{2(2i - 1)}},
$$
it follows that 
 \begin{align*}
 \sum_{n=0}^{\infty}b'(n)q^n & = \sum_{n=0}^\infty|\mathcal P_{2,1}(n)|q^n=\sum_{n=0}^{\infty}|\POD_1(n)|q^n. \end{align*}
 \end{proof}
 
 \begin{proof}[Combinatorial Proof]
Recall that when showing $pod(n) = p_2(n)$, we map $\lambda \in \mathcal{POD}(n)$ to $\mu \in \mathcal{P}_2(n)$ by splitting each part congruent to $2 \bmod 4$ into two equal odd parts and leaving all other parts the same. Thus $b'(n)$ is equal to the number of  parts congruent to $2 \bmod 4$ in all partitions of $\mathcal{POD}(n)$. 

Denote by $\POD^*(n)$ the set of partitions $\POD(n)$ with a single part congruent to $2 \bmod 4$ marked. Note that marked partitions differ from overpartitions since any part congruent to $2 \bmod 4$ may be marked. For example, $(9,7,6^*,6,6,4)$ and  $(9,7,6,6^*,6,4)$ are different marked partitions in $\POD^*(n)$. Thus, $b'(n)=|\POD^*(n)|$. We establish a bijection $f:\POD^*(n)\to \POD_1(n)$. 

Start with $\lambda\in \POD^*(n)$ and suppose $\lambda_i=a^*$ is the marked part. Moreover, suppose $\lambda_{i-j+1}=a$ and $\lambda_{i-j}>a$ (we set $\lambda_0=\infty$). Thus the marked part $\lambda_i$ is the $j^{th}$  part among the parts equal to $a$. We define $f(\lambda)$ to be the partition obtained from $\lambda$ by removing the marking and splitting $j$ of the parts equal to $a$ each  into two equal parts. Since $a\equiv 2\bmod 4$, the new parts are equal and odd. Thus $f(\lambda)\in \POD_1(n)$. 

To show that $f$ is invertible,  start with $\mu\in \POD_1(n)$ with repeated odd part $b$. Let $m(b)\geq 2$ be the multiplicity of $b$ in $\mu$. Then $f^{-1}(\mu)$ is the marked partition obtained from $\mu$ by merging $\lfloor m(b)/2 \rfloor$ pairs of parts equal to $b$ to obtain $\lfloor m(b)/2 \rfloor$ parts equal to $2b$. We mark the $\lfloor m(b)/2 \rfloor^{th}$ part equal to $2b$ in the obtained partition. Thus, $f^{-1}(\mu)\in \POD^*(n)$.

From the bijective proof of $pod(n)=p_2(n)$ we see that 
$b'(n)$ is also equal to the number of odd parts in all partitions in $\mathcal{P}_2(n)$ each counted with (the integer part of) half its multiplicity, i.e. $$b'(n)=\sum_{\lambda \in \mathcal{P}_2(n)}\sum_{c\in \lambda, c \text{ odd}}\Bigl\lfloor \frac{m_{\lambda}(c)}{2}\Bigr\rfloor,$$ wher $m_\lambda(c)$ is the multiplicity of $c$ in $\lambda$. 

Denote by  $\mathcal P^*_2(n)$ the set of partitions in $\mathcal P_2(n)$ with a single odd part marked. Moreover, part  $a$ in $\lambda$ may be marked only if it is  the $j^{th}$ occurrence of $a$ in $\lambda$  with $j$ even. Thus $b'(n)=|\mathcal P^*_2(n)|$. We establish a bijection $g:\mathcal P^*_2(n)\to \mathcal P_{2,1}(n)$. 

Start with $\lambda\in \mathcal P^*_2(n)$ and suppose $\lambda_i=a^*$ is the marked part. Moreover, suppose $\lambda_i$ is the $j^{th}$ occurrence of $a$ in $\lambda$. We define $g(\lambda)$ to be the partition obtained from $\lambda$ by removing the marking and merging $j/2$ pairs of parts equal to $a$ to obtain $j/2$ parts equal to $2a$. Thus $g(\lambda)\in \mathcal P_{2,1}(n)$. 

To show that $g$ is invertible,  start with $\mu\in \mathcal P_{2,1}(n)$ with repeated part $b\equiv 2\bmod 4$. Let $m(b)\geq 2$ be the multiplicity of $b$ in $\mu$. Then $g^{-1}(\mu)$ is the partition obtained from $\mu$ by splitting each part equal to $b$ into two parts equal to $b/2$ and marking the $(2m(b))^{th}$ occurrence of $b/2$ in the obtained partition. Thus  $g^{-1}(\mu)\in  \mathcal P^*_2(n)$.
 \end{proof}


\begin{thebibliography}{00}


 

 
 
\bibitem{A98} 
G. E. Andrews, 
\textit{The Theory of Partitions}, 
Cambridge Mathematical Library, Cambridge University Press, Cambridge, 1998. Reprint of the 1976 original.

\bibitem{A-Beck}	
	G. E. Andrews, 
 \textit{Euler's partition identity and two problems of George Beck}, Math. Student 86 (2017), no. 1-2, 115--119.
 
\bibitem{A09} G. E.  Andrews, \textit{Partitions with distinct evens.} Advances in combinatorial mathematics, 31–37, Springer, Berlin, 2009.

\bibitem{A72} G. E. Andrews,  \textit{Two theorems of Gauss and allied identities proved arithmetically.} Pacific J. Math. 41 (1972), 563--578.


\bibitem{AE} G. E. Andrews and K. Eriksson,  \textit{Integer partitions}. Cambridge University Press, Cambridge, 2004. x+141 pp.

\bibitem{BM21a} C. Ballantine and M.  Merca, \textit{Combinatorial proof of the minimal excludant theorem}. Int. J. Number Theory 17 (2021), no. 8, 1765--1779. 

\bibitem{BM21} C. Ballantine and M. Merca, \textit{$4$-regular partitions and the pod function}, preprint. 

\bibitem{BW21} C. Ballantine, and A. Welch,  \textit{Beck-Type Identities for Euler Pairs of Order r}. In: Bostan, A., Raschel, K. (eds) Transcendence in Algebra, Combinatorics, Geometry and Number Theory. TRANS 2019. Springer Proceedings in Mathematics $\&$ Statistics, vol 373. Springer, Cham. (2021)

\bibitem{BZ85} D. M. Bressoud, David and D. Zeilberger,  \textit{Bijecting Euler's partitions-recurrence. Amer. Math. Monthly 92 (1985), no. 1, 54--55.}

\bibitem{CL}
S. Corteel, J. Lovejoy, 
\textit{Overpartitions}. 
{Trans. Amer. Math. Soc.} {356} (2004) 1623--1635.


\bibitem{FGK} A. Fink, R.K. Guy, and M. Krusemeyer, \textit{Partitions with parts occurring at most thrice}, Contrib. Discrete
Math. 3 (2) (2008) 76--114.

\bibitem{FT20} S. Fu and D. Tang, \textit{On certain unimodal sequences and strict partitions.} Discrete Math. 343 (2020), no. 2, 111650, 8 pp.

\bibitem{Gl} 
J. W. L. Glaisher, 
A theorem in partitions, 
\textit{Messenger of Math.} 12 (1883), 158--170.

\bibitem{G75}
H.\ Gupta,
\textit{Combinatorial proof of a theorem on partitions into an even or odd number of parts}. J. Combinatorial Theory Ser. A 21 (1976), no. 1, 100–103.

\bibitem{KK18} L. W. Kolitsch and  S. Kolitsch, \textit{
A combinatorial proof of Jacobi's triple product identity},
Ramanujan J. 45 (2018), no. 2, 483--489.

\bibitem{M17} M. Merca, \textit{New relations for the number of partitions with distinct even parts.}  J. Number Theory 176 (2017), 1--12.

 \bibitem{S71} M. V. Subbarao, \textit{Partition theorems for Euler pairs}. Proc. Amer. Math. Soc. 28 (1971), 330--336.
 
  \bibitem{B1} The On-Line Encyclopedia of Integer Sequences, oeis: A090867 and A265251.


\end{thebibliography}
\end{document}